\documentclass[12pt,reqno]{amsart}
\usepackage[margin=3cm]{geometry}

\usepackage{mathrsfs}
 \usepackage{graphicx}
 \usepackage[utf8]{inputenc}
 \usepackage[T1]{fontenc}
 \usepackage{lmodern}
 \usepackage[normalem]{ulem}
 \usepackage{verbatim}
 \usepackage{bbm}
 \usepackage{stmaryrd}
 \usepackage{amsmath}
 \usepackage{amssymb}
 \usepackage{dsfont}
\usepackage{amsthm}
\usepackage{hyperref}
\usepackage{relsize}
\usepackage{exscale}
\usepackage{cleveref}
\usepackage[usenames,dvipsnames]{color}
\usepackage{enumerate}

\theoremstyle{plain}
\newtheorem{theo}{Theorem}[section]%

\newtheorem{cor}[theo]{Corollary}%
\newtheorem{lemma}[theo]{Lemma}%
\newtheorem{prop}[theo]{Proposition}%
\newtheorem{claim}[theo]{Claim}%
\theoremstyle{definition}
 \newtheorem{defi}[theo]{Definition}
\theoremstyle{remark}
 \newtheorem{remark}[theo]{Remark}

\numberwithin{equation}{section}

\newcommand\nc\newcommand
\nc\dmo\DeclareMathOperator
\nc\bs\boldsymbol

\nc{\red}[1]{{\color{red} #1}}
\nc{\blue}[1]{{\color{blue} #1}}
\nc{\green}[1]{{\color{green} #1}}
\nc{\cyan}[1]{{\color{cyan} #1}}
\definecolor{purple}{rgb}{0.9,0,0.8}
\nc{\purple}[1]{{\color{purple} #1}}
\definecolor{gray}{rgb}{0.5,0.5,0.5}
\nc{\gray}[1]{{\color{gray} #1}}
\nc{\note}[1]{{\blue{\textup{[#1]}}}}

\nc{\wt}{\widetilde}
\nc{\wh}{\widehat}
\nc{\eqd}{\stackrel{\text{\tiny $d$}}{=}}

\nc{\ii}{\mathrm{i}}

\def \HS {\mathrm{HS}}

\def \op {\mathrm{op}}
\dmo{\rank}{rank}
\def\tr{{\rm Tr}}

\renewcommand{\P}{\mathbb P}
\newcommand{\E}{\mathbb E}
\newcommand{\Var}{\mathbb{V}{\rm ar}}

\newcommand{\C}{\mathbb C}
\newcommand{\R}{\mathbb R}
\newcommand{\N}{\mathbb N}
\newcommand{\Z}{\mathbb Z}

\renewcommand\phi\varphi 

\nc{\cA}{\mathcal{A}}
\nc{\cB}{\mathcal{B}}
\nc{\cC}{\mathcal{C}}
\nc{\cD}{\mathcal{D}}
\nc{\cE}{\mathcal{E}}
\nc{\cF}{\mathcal{F}}
\nc{\cG}{\mathcal{G}}
\nc{\cH}{\mathcal{H}}
\nc{\cI}{\mathcal{I}}
\nc{\cJ}{\mathcal{J}}
\nc{\cK}{\mathcal{K}}
\nc{\cL}{\mathcal{L}}
\nc{\cM}{\mathcal{M}}
\nc{\cN}{\mathcal{N}}
\nc{\cO}{\mathcal{O}}
\nc{\cP}{\mathcal{P}}
\nc{\cQ}{\mathcal{Q}}
\nc{\cR}{\mathcal{R}}
\nc{\cS}{\mathcal{S}}
\nc{\cT}{\mathcal{T}}
\nc{\cU}{\mathcal{U}}
\nc{\cV}{\mathcal{V}}
\nc{\cW}{\mathcal{W}}
\nc{\cX}{\mathcal{X}}
\nc{\cY}{\mathcal{Y}}
\nc{\cZ}{\mathcal{Z}}

\nc{\fkp}{\mathfrak{p}}    \nc{\fkP}{\mathfrak{P}}
\nc{\fkq}{\mathfrak{q}}    \nc{\fkQ}{\mathfrak{Q}}
\nc{\sP}{\mathscr{P}}		

\dmo{\id}{I}
\nc{\eps}{{\varepsilon}}
\nc\ls\lesssim
\nc\gs\gtrsim
\nc\lls{\,\ls\,}
\nc{\ggs}{\,\gs\,}
\dmo{\ball}{\mathbb{B}}
\dmo{\sph}{\mathbb{S}}
\dmo{\dist}{dist}
\nc{\pr}{\P}
\dmo{\Normal}{\cN}

\dmo{\GUE}{GUE}
\nc{\Lin}{{\cL}}
\nc{\Linz}{\Lin^z}
\nc{\slin}{s}
\nc{\rlin}{r}
\nc{\tlin}{{\blue{\slin_0}}}
\nc{\Slin}{S}
\nc{\tSlin}{\wt{S}}
\nc{\Tlin}{{\blue{S_0}}}
\nc{\pol}{{\fkp}}
\nc{\Pol}{{P}}
\nc{\LL}{{\bs L}}
\nc{\MM}{{\bs M}}
\nc{\UU}{{\bs U}}
\nc{\WW}{{\bs W}}
\nc{\QQ}{{\bs Q}}
\nc{\RR}{{\bs R}}
\nc{\hRR}{\wh{\RR}}
\nc{\HH}{{\bs H}}
\nc{\bAA}{{\bs A}}
\nc{\tAA}{\bs{\wt{A}}}
\nc{\bq}{{\bs q}}
\nc{\hUU}{\wh{\UU}}
\nc{\hFlat}{\wh{\Flat}}
\nc{\hU}{\wh{U}}
\nc{\hu}{\wh{u}}
\nc{\hv}{\wh{v}}
\nc{\bv}{{\bs{v}}}
\nc{\hbu}{\wh{\bu}}
\nc{\tS}{\wt{S}}
\nc{\hbq}{\wh{\bq}}
\nc{\LLz}{{\bs L}^z}
\nc{\tU}{\wt{U}}
\nc{\tLL}{\bs{\wt{L}}}	
\nc{\tLLz}{\tLL^z}
\nc{\tMM}{\wt\MM}
\nc{\tM}{\wt{M}}
\nc{\tL}{\wt{L}}
\nc{\tLz}{\tL^z}
\nc{\tUU}{\wt\UU}
\nc{\Le}{L}
\nc{\Lez}{L^z}
\nc{\Kez}{K^z}
\nc{\hL}{\wh{L}}
\nc{\hLz}{\wh{L}^z}
\nc{\num}{{n}}
\nc{\rnk}{{n_0}}
\nc{\Mat}{\mathbb{M}}
\nc{\bx}{{\bs x}}
\nc{\xx}{{\bx}}
\nc{\bX}{{\bs X}}
\nc{\XX}{{\bX}}
\nc{\bg}{{\bs g}}
\nc{\bh}{{\bs h}}
\nc{\bu}{{\bs u}}
\nc{\bxi}{{\bs \xi}}
\nc{\ba}{{\bs a}}
\dmo{\diag}{diag}
\nc{\dd}{{\rnk+1}}
\dmo{\col}{col}
\dmo{\row}{row}
\dmo{\Span}{Span}
\nc{\DD}{{[0,\rnk]}}
\dmo{\Struct}{Struct}
\dmo{\Stief}{\mathbb{U}}
\dmo{\Sm}{S}
\dmo{\tSm}{\wt{S}}
\nc{\smin}{\sigma_{\min}}
\nc{\good}{\cA}
\nc{\badR}{\cB}
\dmo{\smk}{1}
\nc{\shift}{a_0}
\nc{\bound}{B}
\nc{\Flat}{\bs W}
\nc{\pp}{p}
\dmo{\Leb}{Leb}
\nc{\Ylin}{Y}
\dmo{\proj}{Proj}
\nc{\scpol}{{p}}
\nc{\tw}{{\wt{w}}}
\nc{\tv}{\wt{v}}
\nc{\tx}{\wt{x}}
\nc{\pstr}{{\delta}}
\nc{\pnet}{\rho_\star}
\nc{\pssv}{{\alpha}}
\nc{\pflat}{{\beta}}
\nc{\rexp}{\gamma}
\nc{\matt}[3]{\Mat_{#1}^{#2}(#3)}
\nc{\stief}[2]{\Stief_{#1}^{#2}}
\nc{\matball}[2]{\ball_{#1}^{#2}}
\nc{\hw}{\wh{w}}
\nc{\pfine}{\rho_1}

\nc{\subsV}{{\mathsf{V}}}
\nc{\subsW}{{\mathsf{W}}}
\nc{\subsp}{\subsW}
\dmo{\Walk}{Walk}
\nc{\gvec}{\bs{\zeta}}
\nc{\gstep}{\zeta}
\nc{\BALL}{\ball_\star}
\nc{\tbad}{\wt{\badR}}
\nc{\walkshift}{M}

\nc{\nick}[1]{{#1}}
\nc{\nickC}[1]{{\bf \emph{\nick{[#1]}}}}
\nc{\alice}[1]{{\color{red} #1}}
\nc{\jonathan}[1]{{\color{blue} #1}}

\title[Pseudospectrum for quadratic polynomials of Ginibre matrices]{Spectrum and pseudospectrum for quadratic polynomials in Ginibre matrices}

\author[N.\ Cook]{Nicholas Cook$^*$}\thanks{${}^*$Partially supported by NSF grant DMS-1606310}
\address{
Duke University, Durham, NC 27708, USA}
\email{nickcook@math.duke.edu}
\author[A.\ Guionnet]{Alice Guionnet$^\dagger$}\thanks{${}^\dagger$Partially supported by  Labex MILYON/ANR-10-LABX-0070 }
\address{Universit\'e de Lyon, ENSL, CNRS, France }
\email{Alice.Guionnet@ens-lyon.fr}
\author[J.\ Husson]{Jonathan Husson}
\address{Universit\'e de Lyon, ENSL,  France}
\email{Jonathan.Husson@ens-lyon.fr}

\date{\today}

\begin{document}

\begin{abstract}
For a fixed quadratic polynomial $\pol$  in $n$ non-commuting variables, and $n$ independent $N\times N$ complex Ginibre matrices $X_1^N,\dots, X_n^N$, we establish the convergence of the empirical spectral distribution of $P^N =\pol(X_1^N,\dots, X_n^N)$ to the Brown measure of $\pol$ evaluated at $n$ freely independent circular elements $c_1,\dots, c_n$ in a non-commutative probability space. 
The main step of the proof is to obtain quantitative control on the pseudospectrum of $P^N$.
Via the well-known linearization trick this hinges on anti-concentration properties for certain matrix-valued random walks,
which we find can fail for structural reasons 
of a different nature from the arithmetic obstructions that were illuminated in works on the Littlewood–Offord problem for discrete scalar random walks.
\end{abstract}


\maketitle
\setcounter{tocdepth}{2}

\section{Introduction}

Recall that for an $N\times N$ matrix $A$ with complex entries and complex eigenvalues $\lambda_1(A),\dots, \lambda_N(A)$ (not necessarily distinct), the \emph{empirical spectral distribution (ESD)} is the probability measure
\begin{equation}	\label{def:ESD}
\mu_A := \frac1N \sum_{j=1}^N \delta_{\lambda_j(A)}.
\end{equation}
For an ensemble of random $N\times N$ matrices $(A^N)_{N\ge 1}$, a central problem in random matrix theory is to establish a law of large numbers for the ESDs -- that is, to prove that (in the vague topology) $\mu_{A^N}$ converges in probability to some deterministic probability measure $\mu$. 
The seminal works of Wigner \cite{Wig58} and Marchenko--Pastur \cite{pastur-marchenko} addressed this problem for matrices with i.i.d.\ entries above the diagonal (Wigner ensembles) and Gram matrices for i.i.d.\ rectangular matrices using the moment and Stieltjes transform methods.

For non-Hermitian \emph{i.i.d.\ ensembles} $X^N$, having i.i.d.\ entries and no symmetry assumption, the problem was only addressed at a comparable level of generality much more recently in \cite{TaVu:circ,GT,TVUniversality}, where it was shown that $\mu_{\frac1{\sqrt{N}}X^N}$ converges to Girko's circular law $\mu_{\text{circ}}$, the uniform measure on the unit disk.  The moment and Stieltjes transform methods were insufficient to establish the circular law due to the instability of the spectrum of non-normal matrices; these obstacles were overcome using tools from additive combinatorics and geometric functional analysis. We refer to the survey \cite{BoCh:survey} for further background.

The circular law was recently  generalized to the non-homogeneous case where the variance of the entries depends on the site \cite{chjr,AEK-inhom} and further to the case of correlated entries provided their correlations decay fast enough \cite{AK}. The case of matrices invariant under left multiplication by Haar unitary matrices led to the single ring theorem proven in \cite{SingleRing}; see also \cite{FZ} for a non-rigorous derivation. 
The sum of i.i.d.\ Haar unitary, orthogonal and permutation matrices were treated in \cite{BaDe, BCZ}.

Random matrices have played an important role in the theory of operator algebras since the seminal work of Voiculescu \cite{V91}. In these applications one is interested not in a single random matrix $X^N$ but rather a collection $X_1^N,\dots, X_n^N$ of a bounded number of independent random matrices, and the algebra they generate in the large $N$ limit.
For self-adjoint polynomials in several independent Wigner matrices, Voiculescu showed that the theory of  free probability gives powerful tools to understand the limit of ESDs for a non-commutative polynomial in those matrices \cite{V91}. Indeed, if $W^N_1,...,W^N_n$ are independent $N\times N$ Wigner matrices and $\pol$ is a non-commutative self-adjoint polynomial, then the ESD of $\pol(W^N_1,...,W^N_n)$ converges towards the spectral distribution of $\pol(s_1,...,s_n)$, where $s_1,...,s_n$ are $n$ freely independent semi-circular elements of a von Neumann algebra.

\begin{figure}
\includegraphics[width=200pt]{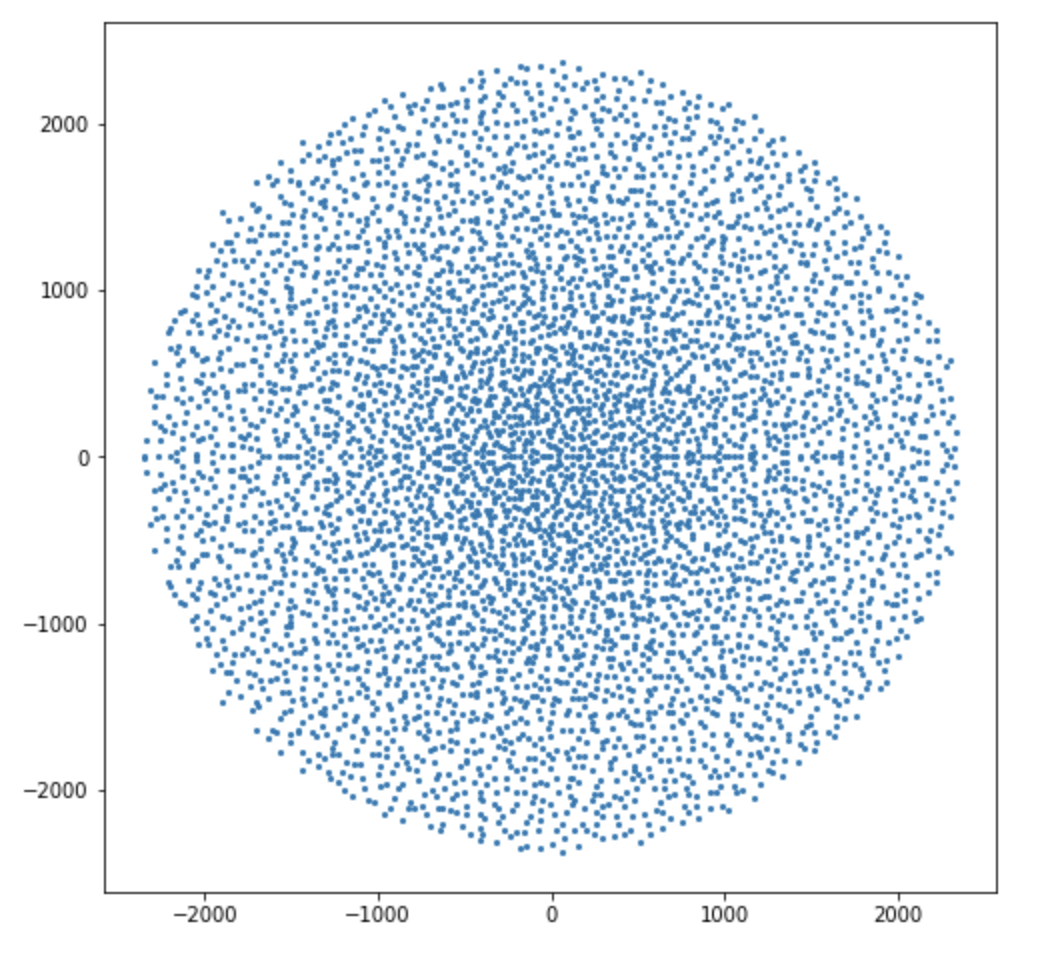}
\includegraphics[width=195pt]{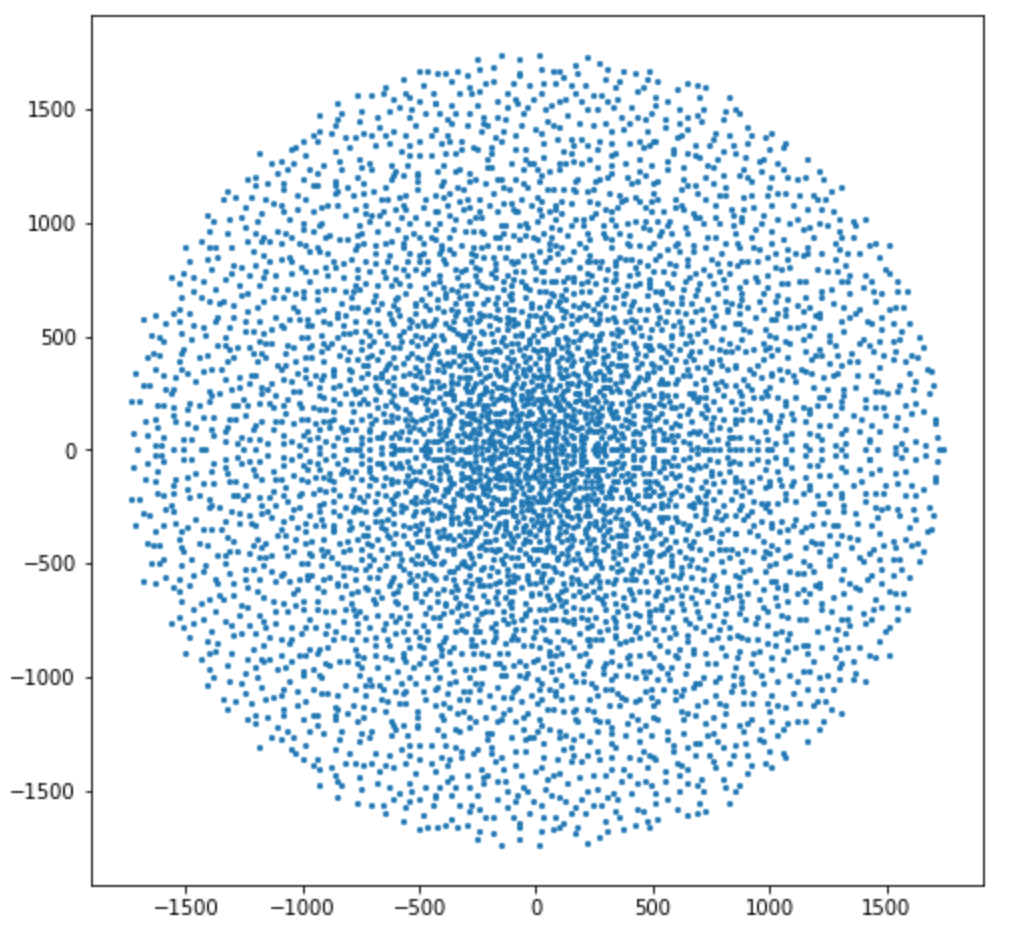}
\caption{Simulated spectra of $XY+YX$ (left) and $XY-0.3YZ + 0.1ZX$ (right) for independent $5000\times 5000$ matrices $X,Y,Z$ with entries independently and uniformly drawn from $[-1,1]$.}
\end{figure}

The problem is more involved in the case of a non-Hermitian polynomial since the convergence in $*$-moments does not yield the convergence of the ESDs. 
In this case, the analogue of the spectral distribution for a non-normal element of a von Neumann algebra is the Brown measure \cite{Brown}, and the result one can expect for usual matrix models is the convergence of the empirical measure toward the Brown measure of 
the $*$-moments limit. Even the computation of the candidate Brown measure limit is a non-trivial task:
there has been recent progress by Speicher, Mai, Belinschi and Sniady, who found an algorithm that gives such Brown measures using linearization techniques and subordination results \cite{BMS,BSS}; see also \cite{BL} for the computation of some specific Brown measures. 

The convergence of ESDs of polynomials in independent matrices was so far tackled  only in specific cases.  The product of independent Ginibre matrices was studied by  F. G\"otze, A. Naumov and A. Tikhomirov \cite{Tik,GT} and S. O'Rourke and A. Soshnikov \cite{OS}, as well as the sum of such products \cite{Kti}. For products of Girko's elliptic random matrices see \cite{ORSV}. Yet, there are no general results for the convergence of  the ESDs of non-self-adjoint polynomials in several independent matrices.

\subsection{Spectral convergence and the pseudospectrum}

It turns out that the qualitative problem of establishing convergence of spectral measures to the Brown measure is intimately related to quantitative (finite-$N$) questions of interest in numerical analysis.
The key difficulty for non-normal matrices lies in the instability of their eigenvalues: it is well known that for certain matrices even a tiny perturbation of a single entry can drastically change the spectrum.

One may expect however that \emph{random} non-normal matrices have a more stable spectrum. 
\nick{A long line of works beginning with \cite{Sniady} has shown that in many cases, for a sequence of non-normal matrices converging in star moments, it is sufficient to perturb by a small random matrix (of vanishing norm) in order to ``regularize'' the ESDs and guarantee convergence to the Brown measure -- see \cite{Sniady, GWZ,Wood:noise, BPZ,VoZe} and references therein. This regularizing effect has been exploited for applications in numerical analysis in the recent works \cite{BKMS,BGKS}.}
Yet, these results are based on independence of the entries of random matrices and there are, as quoted above, no  general results when  the entries  start to be strongly coupled, for instance for  the commutator of two  independent Ginibre matrices.

We briefly sketch how quantitative measures of spectral instability arise in the study of limiting spectral distributions; 
a more formal discussion is deferred to \Cref{sec:brown}.
The starting point is to note that the empirical spectral distribution (ESD) $\mu_A$ of an $N\times N$ matrix $A$ can be recovered from the Laplacian of the log-modulus of the characteristic polynomial: recalling \eqref{def:ESD}, we have
\[
\mu_A =
\frac1{2\pi N} \sum_{i=1}^N \Delta_z  \log|\lambda_i(A) - z| =  \frac1{2\pi N} \Delta_z\log |\det(A-z)|.
\]
On the other hand, the log-modulus of the characteristic polynomial can also be expressed 
\[
\frac1N\log|\det(A-z)| = \frac1{2N} \log \det [(A-z)(A-z)^*]  = \frac12 \int_0^\infty \log x \,d\mu_{(A-z)(A-z)^*}(x)
\]
and so we have the identity
\begin{equation}	\label{mu.identity}
\mu_A = \frac1{4\pi}  \Delta_z\int_0^\infty \log x \,d\mu_{(A-z)(A-z)^*}(x)
\end{equation}
expressing the ESD of a (possibly non-normal) matrix $A$ in terms of the ESDs of the collection of Hermitian matrices $(A-z)(A-z)^*$ with $z\in \C$. 

Now for a sequence $A^N$ of $N\times N$ matrices converging in $*$-moments to an element $a$ of a von Neumann algebra one has convergence of the ESDs $\mu_{(A^N-z)(A^N-z)^*}$ to the spectral measure $\mu_{(a-z)(a-z)^*}$, and so one might hope to have convergence of the ESDs $\mu_{A^N}$ to the measure $\mu_a$ obtained by substituting $a$ for $A$ on the right hand side of \eqref{mu.identity}.
This is not true in general: we have already mentioned that convergence in $*$-moments does not guarantee convergence of the spectral distributions, and indeed one notes that the hoped-for identity fails if there is escape of mass of $\mu_{(A^N-z)(A^N-z)^*}$ to the singularities of the logarithm at $0$ and $+\infty$. 
However, we may take \eqref{mu.identity} as a reasonable guess, and indeed this amounts to the definition of the Brown measure $\nu_a$ (reviewed in \Cref{sec:brown} below). 

We thus see that in order to access the limiting measure via \eqref{mu.identity} it is crucial to control the largest and smallest eigenvalues of $(A^N-z)(A^N-z)^*$. For random matrices it turns out that the more delicate task is to bound the smallest eigenvalue from below.
Recall that the $\eps$-pseudospectrum of a square matrix $A$ is the set
\begin{equation}	\label{def:pseudospectrum}
\Lambda_\eps(A) = \{ z\in \C: \smin(A-z) \le \eps\} 
\end{equation}
where $\smin(A) = \sqrt{\lambda_{\min}(AA^*)}$ denotes the smallest singular value of a matrix $A$.
Alternatively, $\Lambda_\eps(A)$ is the union of the spectra of $A+E$ over all perturbations $E$ of spectral norm at most $\eps$ (we refer to \cite[Chapter 2]{TrEm} for the demonstration of this equivalence).
The pseudospectrum is an important object in numerical analysis that quantifies the stability of the spectrum under small perturbations.  
We refer to the textbook \cite{TrEm} for further background.
To prove convergence of the ESDs $\mu_{A^N}$ one must show that $\Lambda_\eps(A^N)$ is  asymptotically null for $\eps =\eps(N) \to 0$ not too fast (any polynomial order will be sufficient). 

We point out that the problem of bounding $\sup_{z\in \C}\pr\{\smin(A-z)\le \eps\}$ for a random $N\times N$ matrix $A$ generalizes the well-studied \emph{anti-concentration problem} for scalar random variables, which is the case $N=1$. Moreover, scalar concentration inequalities have played a fundamental role in the study of invertibility of random matrices. See \Cref{sec:anti} for further discussion.

To establish uniform integrability of the logarithm in \eqref{mu.identity} also requires some control on the $k$-th smallest singular value of $A^N-z$ for moderately small values of $k$ (in the range $[N^{-c}, \delta N  ]$ for small fixed $c,\delta>0$). 
For a polynomial $P^N$ in Ginibre matrices
this can be done thanks to local laws which can be found in this context in \cite{HT}. Hence, in the present work the main issue is to obtain control on the pseudospectrum of $P^N$.

\subsection{Main results}
\label{sec:main}

In the sequel we write $\C\langle x_1,\dots, x_n\rangle$ for the set of polynomials with complex coefficients in $n$ non-commuting indeterminates $x_1,\dots x_n$.  
We recall the following definition of the complex Ginibre ensemble: 

\begin{defi}
A random matrix $X=X^N$ is an $N\times N $ complex Ginibre matrix if the families $(\sqrt{N} \Re (X_{i,j}))_{i,j \in[N]}$ and $(\sqrt{N} \Im (X_{i,j}))_{i,j \in[N]}$ are independent i.i.d.\ families of random variables of law $\Normal(0,1/2)$.
\end{defi}

Our first main result establishes convergence of the empirical spectral distribution of $\pol(X_1^N,\dots, X_n^N)$ 
for any quadratic non-commutative polynomial $\pol$. 

\vspace{.2cm}

\begin{theo}[Convergence to the Brown measure] \label{thm:brown}
Let $n\in \N$ and $\pol\in \C\langle x_1,\dots, x_n\rangle$ be a non-commutative polynomial of degree two.
For each $N\in \N$ let $X_1^N,\dots,X_n^N$ be $n$ independent $N\times N$ complex Ginibre matrices and set $\Pol^N = \pol(X_1^N,\dots, X_n^N)$. 
Then
\[
\mu_{\Pol^N} \to \nu_{\pol(c_1,\dots,c_n)} 
\]
weakly in probability, where $c_1,\dots, c_n$ are $*$-free circular elements of a $W^*$-probability space 
and $\nu_{\pol(c_1,\dots,c_n)}$ is the Brown measure of  $\pol(c_1,\dots,c_n)$ (defined in \Cref{sec:brown} below). 
\end{theo}

\vspace{.2cm}

As we described above, the main step for proving \Cref{thm:brown} is to control the pseudospectrum of $P^N$, which is accomplished in our second main result. 
Whereas \Cref{thm:brown} was stated for a sequence of matrices $P^N$ of growing size, the following is a non-asymptotic result for matrices of any fixed size. 

\vspace{.1cm}

\begin{theo}[Control on the pseudospectrum] \label{thm:pseudo}
Let $N,n\in \N$ and $\pol\in \C\langle x_1,\dots, x_n\rangle$ be a non-commutative polynomial of degree two.
There is an absolute constant $C_0>0$ and $C(\pol),c(\pol)>0$ depending only on $\pol$ such that the following holds. 
Let $X_1,\dots,X_n$ be $n$ independent $N\times N$ complex Ginibre matrices and set $\Pol = \pol(X_1,\dots, X_n)$. 
For any $z \in  \C$ and any $\eps>0$,
\[ 
\P \big\{ \sigma_{\min}(\Pol - z) \le \eps \big\} \le C(N^{C_0}\eps^{c} + e^{-N}).
\]
\end{theo}

\vspace{.1cm}

\begin{remark}	\label{rmk:constants}
Our proof shows one can take $C_0=13/3$ but we have not tried to optimize this constant. 
We obtain the dependence $c(\pol) = 1/(\rnk+1)$, where $\rnk$ is the rank of the quadratic form associated to the homogeneous degree-two part of $\pol$ (see \Cref{lem:linearize}). 
The constant $C$ depends on $\pol$ only through the rank parameter $\rnk$ and the size of the coefficients, quantified by the norm of the matrix $(\slin_0,\dots, \slin_\rnk)$ with columns $\slin_k$ as in \Cref{cor:linearize}. The dependence on these two parameters is polynomial, but we have not tracked the precise order.
Finally, we believe the exponential error term $e^{-N}$ should not be necessary under the assumption of Gaussian entries, as in the degree-one case (cf.\ \Cref{lem:smin}). This term is due to our reliance on net arguments, and we expect that removing it would require a different approach. 
\end{remark}

From the Fubini--Tonelli theorem and Markov's inequality we have the following corollary on the density of the pseudospectrum (recall \eqref{def:pseudospectrum}).

\vspace{.1cm}

\begin{cor}	\label{cor:pseudo}
With hypotheses in \Cref{thm:pseudo}, for any Borel set $\Omega\subset \C$ and any $\eps>0$ we have
\[
\E \Leb (\Lambda_\eps(P) \cap \Omega) \le C(N^{C_0}\eps^{c} + e^{-N})\Leb(\Omega)
\]
where $\Leb$ denotes the Lebesgue measure on $\C$. In particular, for any $K>0$ there exists $C(K,\pol)>0$ such that for all $N$ sufficiently large depending on $\pol$,
for any fixed Borel set $\Omega$ of finite measure we have that with probability at least $1-N^{-100}$,
\[
\Leb (\Lambda_{N^{-C}}(P) \cap \Omega)\le N^{-K}\Leb(\Omega) .
\]
\end{cor}

\vspace{.3cm}

For the proof of \Cref{thm:pseudo} we make use of the so-called linearization trick incepted by Haagerup and Thorbj\o rnsen in \cite{HT} to reduce the problem to the estimate of the smallest singular value of a larger block matrix with coefficients given by  linear combinations of  the Ginibre matrices, hence obtaining a matrix with strongly dependent entries.  
Such a matrix can be viewed as an $N\times N$ matrix whose entries are independent random matrices of bounded size, and which themselves have correlated entries. Where previous works on the pseudospectrum of i.i.d.\ matrices reduced the problem to establishing anti-concentration for a scalar random walk, here we are led to study anti-concentration of \emph{matrix} random walks and their determinants. We defer further discussion of proof ideas to \Cref{sec:outline}.

\subsection{Anti-concentration for scalar polynomials} \label{sec:anti}

The problem addressed by \Cref{thm:pseudo} of bounding $\pr\{\smin(\Pol-z)\le \eps\}$ uniformly over $z\in \C$ reduces in the case $N=1$ to the classical anti-concentration problem for quadratic polynomials. Under the Gaussian assumption our result in this case (without the error term $e^{-N}$) is
an easy consequence of the Carbery--Wright inequality, which is a basic ingredient of the proof -- cf.\ \Cref{lem:anti-poly}.

Anti-concentration for \emph{linear} forms in scalar random variables of general (especially discrete)  distribution has been extensively explored since 
the work of Littlewood and Offord \cite{LiOf}, and its connection to the invertibility of random matrices was made by Koml\'os \cite{Komlos67}. Within the \emph{Littlewood--Offord theory}, a key theme developed in works of Hal\'asz \cite{Halasz}, Tao--Vu \cite{TaVu:sing, TaVu:cond}, Rudelson--Vershynin \cite{RuVe:LO} and Nguyen--Vu \cite{NgVu:ILO} 
is the link between anti-concentration of linear forms in discrete random variables and a lack of arithmetic structure in the coefficients.

The theory of anti-concentration for higher degree polynomials is still under development. 
Costello, Tao and Vu made a link between the invertibility of symmetric random matrices and anti-concentration of scalar quadratic forms \cite{CTV}.
Further advances on the invertibility of symmetric random matrices have appeared in \cite{Vershynin:symmetric, FeJa:symmetric, CMMM}, though we remark that all of these works make use of a bilinear reduction in order to apply the linear Littlewood--Offord theory.
Works of Costello \cite{Costello}, Nguyen \cite{Nguyen:symmetric,Nguyen:quadratic} and recently of Kwan and Sauermann \cite{KwSa} have shown that for quadratic polynomials in Bernoulli variables, in addition to arithmetic effects, there can be \emph{algebraic} structural reasons for a lack of anti-concentration -- basically, that the polynomial is close to a quadratic form of low rank.

\Cref{thm:pseudo} can be viewed as an anti-concentration result for non-commutative random variables. Analogously to the works \cite{Nguyen:symmetric,Nguyen:quadratic,KwSa} on scalar quadratic forms, we find there are new algebraic issues that are not present in the case of a single shifted i.i.d.\ matrix. As noted in \Cref{rmk:constants}, the rank of the associated quadratic form also plays a role here, though there are numerous other structural issues to be dealt with.
In order to focus on these novel structural challenges we consider here the case of Gaussian entries; an extension to general, possibly discrete distributions would involve a combination of algebraic and arithmetic structural considerations and is beyond the scope of this work.

\subsection{The Brown measure limit}
\label{sec:brown}

In this subsection we define the Brown measure appearing as the limit in \Cref{thm:brown}.
For further background on free probability we refer to \cite[Chapter 5]{AGZ} and \cite[Chapter 11]{MiSp}.

Recall that a (tracial) $W^*$-probability space $(\cA,\tau)$ is a von Neumann algebra $\cA$ equipped with a tracial faithful normal state $\tau$.
According to  Voiculescu's central result \cite{V91}, the non-commutative distribution of $(X_1^N,\dots, X_n^N)$ converges in $*$-moments towards the non-commutative distribution of $n$ free circular elements $(c_{1},\ldots,c_{n})$ in $\cA$ in the sense that for any polynomial $\pol\in \C\langle x_1,\dots, x_n,x_{1}^{*},\ldots,x_{n}^{*}\rangle$
\[
\lim_{N\rightarrow\infty }\frac{1}{N}\tr\left(\pol(X_1^N,\dots, X_n^N,(X_1^N)^{*},\dots, (X_n^N)^{*})\right)=\tau\left(\pol(c_{1},\dots,c_{n},c_{1}^{*},\dots,c_{n}^{*})\right)\qquad a.s.
\]
The  right hand side is a linear map on the set of polynomials in  $\C\langle x_1,\dots, x_n,x_{1}^{*},\ldots,x_{n}^{*}\rangle$ which is uniquely defined  by  its value on monomials given, for any choices of $i_{1},\ldots, i_{k}\in [1,n]$  and any $\eps_1,\dots, \eps_k\in\{1,*\}$, by setting that
$\tau\left(c_{i_{1}}^{\varepsilon_{1}}\cdots c_{i_{k}}^{\varepsilon_{k}}\right)$
is the number of non-crossing pair partitions of  $\{i_{1},\ldots,i_{k}\}$ so that each block $b=(b_{1},b_{2})$ is such  that $i_{b_{1}}=i_{b_{2}}$ and   $(\varepsilon_{b_{1}}, \varepsilon_{b_{2}})=(1,*)$ or $(*,1)$.
By density of polynomial functions in the set of continuous functions,  and the fact that the Ginibre matrices are bounded with high probability, we  see that the latter implies that
for any  bounded continuous function  $f$
\[
\lim_{N\rightarrow\infty }\frac{1}{N}\tr\left(f\left(\pol(X_1^N,\dots, X_n^N,(X_1^N)^{*},\dots, (X_n^N)^{*})\right)\right)=\tau\left(f(\pol(c_{1},\dots,c_{n},c_{1}^{*},\dots,c_{n}^{*}))\right)\quad a.s.
\]
If $\pol$ is  self-adjoint, this guarantees the convergence of the empirical spectral distribution of $\pol\left(X_1^N,\dots, X_n^N,(X_1^N)^{*},\dots, (X_n^N)^{*}\right)$ towards the distribution of 
$\pol(c_{1},\dots,c_{n},c_{1}^{*},\dots,c_{n}^{*})$ which is specified by its moments as given above. If $\pol$ is non-self-adjoint, there is no such simple  relation between the eigenvalues and moments. Girko's idea \cite{Girko} to overcome this difficulty is based on Green's formula  which states that for any twice continuously and  bounded function $\psi$, any complex numbers $\lambda_i, 1\le i\le N$, 
\begin{equation}\label{Green}\frac{1}{N}\sum_{i=1}^N \psi(\lambda_i)=\frac{1}{2\pi}\int \Delta\psi(z)\frac{1}{N} \sum_{i=1}^N \log |z-\lambda_i| dz\end{equation}
where $dz$ denotes the Lebesgue measure on $\mathbb C$.  Taking the $\lambda_{i},1\le i\le N$ to be  the eigenvalues of $\pol\left(X_1^N,\dots, X_n^N,(X_1^N)^{*},\dots, (X_n^N)^{*}\right)$, we deduce
\[
\frac{1}{N}\sum_{i=1}^N \psi(\lambda_i)=\frac{1}{4\pi}\int \Delta\psi(z)\frac{1}{N} \tr\left(\log \left|z- \pol(X_1^N,\dots, X_n^N,(X_1^N)^{*},\dots, (X_n^N)^{*})\right|^{2}\right)dz\,.
\]
Noting that $\left|z- \pol(X_1^N,\dots, X_n^N,(X_1^N)^{*},\dots, (X_n^N)^{*})\right|^{2}$ is  a self-adjoint polynomial in  $X_1^N,\dots, X_n^N,$ $(X_1^N)^{*},\dots, (X_n^N)^{*},$ and 
neglecting the singularity and unboundedness of the logarithm, we are hence prompted to conjecture that the empirical measure  of the eigenvalues of $ \pol(X_1^N,\dots, X_n^N,(X_1^N)^{*},\dots, (X_n^N)^{*})$ converges towards its Brown measure given for any twice continuously differentiable function $\psi$ on $\mathbb C$ by
\[
\int_{\mathbb C} \psi(\lambda)d\nu_{\pol(c_1,\ldots,c_n,c_1^*,\ldots,c_n^*)}(\lambda):=\frac{1}{4\pi} \int  \Delta\psi(z) \tau\left( \log  \left|z-\pol(c_{1},\dots,c_{n},c_{1}^{*},\cdots,c_{n}^{*})\right|^{2}\right)dz.
\]
Note  that  the right hand side makes sense  as  soon as $\Delta\psi$  is nonnegative, taking at worst the value  $-\infty$, since the circular elements  $c_{1},\ldots,c_{n}$ are bounded.

\subsection{Organization of the paper}
\label{sec:organization}

The proof of \Cref{thm:pseudo} occupies the bulk of the paper; the deduction of \Cref{thm:brown} from \Cref{thm:pseudo} following the by-now standard Hermitization procedure is deferred to \Cref{sec:law}. 
In \Cref{sec:outline} we give an informal overview of the proof ideas for \Cref{thm:pseudo} using the example of the anti-commutator $P=X_1X_2+X_2X_1$ in  two Ginibre matrices. 
The formal proofs of Theorems \ref{thm:pseudo} and \ref{thm:brown} occupy the remainder of the paper. 
\Cref{sec:notation} summarizes our notational conventions and records a few standard facts.
The proof of \Cref{thm:pseudo} occupies Sections \ref{sec:linearization}--\ref{sec:struct}.

\section{Proof outline}
\label{sec:outline}

The proof of \Cref{thm:pseudo} proceeds in four steps, carried out in Sections \ref{sec:linearization}--\ref{sec:struct}. 
Here we illustrate the main ideas for the special case of the \emph{anti-commutator} $P=X_1X_2+X_2X_1$ of independent Ginibre  matrices $X_1,X_2$ (the same arguments apply with trivial modifications to the commutator $X_1X_2-X_2X_1$).
Note that Anderson \cite{Anderson} showed the convergence of the empirical measure of the eigenvalues of the anti-commutator $X_{1}X_{2}+X_{2}X_{1}$ in two Wigner matrices, which is self-adjoint.
For our conventions on asymptotic notation we refer to \Cref{sec:notation} below;
in the simple setting of the anti-commutator we are able to avoid the more burdensome notation for block matrices described there that is useful for treating general polynomials.

\subsection{A word on the problem for a single i.i.d.\ matrix}	\label{sec:iid.word}

Previous works on the invertibility and pseudospectrum of a single i.i.d.\ matrix $X$ proceed by reduction to the task of bounding the probability that some column of $X$, say the $j$th column, which we will denote by $\col_j(X)$,  is close to the span $\subsV_{(j)}$ of the remaining $N-1$ columns. The independence of the columns allows one to condition on the remaining columns and focus on controlling the distance of a random vector to a fixed subspace. When the entries are i.i.d.\ Gaussian, by rotational invariance one may take the fixed space to be the span of the standard basis vectors $e_2,\dots, e_N$, and the distance is simply the magnitude of the first coordinate of $\col_j(X)$. The necessary control then follows from the boundedness of the Gaussian density.

For non-Gaussian entries, and in particular for discrete distributions, the problem is more complicated as the small ball probability for the distance depends on the position of the random hyperplane spanned by the remaining columns -- in particular on arithmetic structure in a normal vector to the hyperplane. Such issues will not arise in the present work and we refer the interested reader to \cite{RuVe:LO}.
By taking the entries of our matrices to be Gaussian we can focus on the novel structural pathologies that can arise for algebraic reasons related to the form of the polynomial $\pol$.

\subsection{Step 1: Linearization}

When considering distances of columns to the span of remaining columns for the anti-commutator $\Pol =X_1X_2+X_2X_1$ in Ginibre matrices, we immediately encounter the problem of complete lack of independence of the entries. However, we can retain some independence by using the 
linearization trick of Haagerup and Thorbj\o rnsen \cite{HT}:
one verifies with the Schur complement formula that $(P-z)^{-1}$ is the top-left $N\times N$ block of $(\LLz)^{-1}$, where $\LLz$ is the $3N\times 3N$ block matrix
\begin{equation}	\label{outline:LLz}
\LLz = \begin{pmatrix} -z &X_1 & X_2\\ X_2 & -\id & 0\\ X_1 & 0 & -\id \end{pmatrix}.
\end{equation}
In particular we have the deterministic bound
\begin{equation}
\|(\Pol-z)^{-1}\|_\op \le \|(\LLz)^{-1}\|_\op
\end{equation}
and so it suffices to show
\begin{equation}
\pr \{ \smin(\LLz) \le \eps\} \ls N^C\eps^c + e^{-N}.
\end{equation}
The advantage of this new problem is that $\LLz$ can be viewed as an $N\times N$ matrix with entries that are independent $3\times3$ random matrices: 
\begin{equation}	\label{outline:entries}
\Lez_{i,j}=
\begin{pmatrix} -z \delta_{i,j} & X_1(i,j) & X_2(i,j)\\ X_2(i,j) & -\delta_{i,j} & 0 \\ X_1(i,j) & 0 & -\delta_{i,j} \end{pmatrix}\;, \qquad i,j\in [N].
\end{equation}
(To avoid ambiguity in the proof we will actually write $\tLLz$ for the element of $\Mat_N(\Mat_3(\C))$ associated to $\Pol$, whereas we write $\LLz$ for the associated element of $\Mat_3(\Mat_N(\C))$ as represented in \eqref{outline:LLz}, but we avoid this notation here.)

\subsection{Step 2: Reduction to a random matrix of bounded dimension}

Now we write $\hLz_j$ for $3\times 3$ matrix obtained by projecting the three columns of $\col_j(\LLz) =: (\Lez_{i,j})_{i=1}^N\in \Mat_3(\C)^N\cong \Mat_{3N,3}(\C)$ to the orthogonal complement of the span $\subsV^z_{(j)}$ of the remaining $3N-3$ columns. Specifically, we let $\UU_j\in \Mat_{3N,3}(\C)$ have columns that are orthonormal in $(\subsV^z_{(j)})^\perp$, and put 
\begin{equation}	\label{outline:matwalk}
\hLz_j = \UU_j^* \col_j(\LLz) = \sum_{i=1}^N U_{i,j}^* \Lez_{i,j}
\end{equation}
where $U_{i,j}\in \Mat_3(\C)$ are the corresponding $3\times 3$ blocks of $\UU_j$. 
(We remark that, a posteriori, with high probability $\subsV^z_{(j)}$ has dimension 3 and the columns of $\UU_j$ are in fact a basis.)
The key is that $\col_j(\LLz)$ is independent of $\subsV^z_{(j)}$, and hence $\UU_j$ can be chosen independently of $\col_j(\LLz)$. 

Now on the event that $\smin(\LLz)\le \eps$, there exists a unit vector $\bv=(v_j)_{j=1}^N\in (\C^3)^N$ such that 
\[
\eps \ge \| \LLz \bv \|_2 = \Big\| \sum_{j=1}^N \col_j(\LLz) v_j \Big\|_2.
\]
Since $\bv$ has unit norm there must be some $j_0\in [N]$ such that $\|v_{j_0}\|_2 \ge N^{-1/2}$. Projecting to the orthocomplement of $V_{(j_0)}$ we obtain
\[
\eps  \ge  \Big\| \sum_{j=1}^N \UU_{j_0}^*\col_j(\LLz) v_j \Big\|_2 =  \| \UU_{j_0}^* \col_{j_0}(\LLz) v_{j_0} \|_2
= \| \hLz_{j_0} v_{j_0}\|_2 \ge N^{-1/2} \smin(\hLz_{j_0}). 
\]
After applying a union bound to fix $j_0$ (and using symmetry to assume $j_0=1$) we reduce to showing
\begin{equation}
\pr\{\smin(\hLz_1) \le \eps\} \ls N^C\eps^c+e^{-N}
\end{equation}
for an adjusted constant $C$.

\subsection{Step 3: Gaussian polynomial anti-concentration}

It will be convenient to instead bound the lower tail of $|\det(\hLz_1)|= \prod_{k=1}^3 \sigma_k(\hLz_1)$ rather than the smallest singular value. Since $\hLz_1$ has  bounded norm $O(1)$ with probability $1-O(e^{-N})$ (a property inherited from $\LLz$, which in turn inherits this from $X_1$ and $X_2$ via the triangle inequality), we have $|\det(\hLz_1)|\ls \smin(\hLz_1)$, so it suffices to show
\begin{equation}	\label{outline:goal3}
\pr\{ |\det(\hLz_1)|\le \eps\} \ls N^C \eps^c + e^{-N}.
\end{equation}
We condition on an arbitrary realization of  $\UU_1$ (recall it is independent of $\col_1(\LLz)$).
From \eqref{outline:matwalk} we can view the random matrix $\hLz_1$ as a random walk in $\Mat_3(\C)$ with independent steps 
$U_{i,1}^*\Lez_{i,1}$, 
and our aim is to control the probability that this walk lands in a small neighborhood of the codimension-one variety of singular matrices. 

For a single i.i.d.\ matrix $X$ the analogous problem of bounding
$
\pr\{ |u^* \col_1(X)|\le \eps\}
$
is easily handled when the entries $X(i,j)$ have bounded density, as one can condition on all components but some $j_0$ for which $u_{j_0}$ is not too small. 
In the present setting we cannot reduce so easily to consideration of a single step of the walk. 
Indeed, \eqref{outline:goal3} fails to hold for general $\UU_1$: 
consider for instance the case all of the $U_{i,1}$ are zero except for a single $U_{i_0,1}$ with $i_0\ge2$, in which case $\det(\hLz_1)=\det(U_{i_0,1})\det(\Le_{i_0,1}) \equiv 0$.
On the other hand, it seems unlikely that the random orthonormal set $\UU_1$ in $(\subsV^z_{(1)})^\perp$ would concentrate on a single submatrix $U_{i_0,1}$.

We highlight in particular a key challenge  for establishing \eqref{outline:goal3}: one notes from \eqref{outline:entries} that the random matrices $\Lez_{i,1}$ are themselves singular for $i\ge 2$.
Thus, the invertibility of $\hLz_1$ cannot rely on the randomness of the $\Lez_{i,1}$ alone, but must also come from the geometry of the sequence $(U_{1,1},\dots, U_{N,1})$. 

We proceed to identify sufficient structural conditions on the matrix $\UU_1$ in order to have \eqref{outline:goal3}.
Let us denote the columns of $U_{i,1}^*$ by $v_i^0, v_i^1,v_i^2\in \C^3$. 
For compactness of notation we write $\xi_i^a= X_a(i,1)$, $a=1,2$, for the entries of $X_1,X_2$ that enter in $\col_1(\LLz)$. 
For simplicity of exposition we ignore here the deterministic shift in $\Lez_{1,1}$ and focus on the random walk
\[
W = \sum_{i=1}^N  U_{i,1}^* \Le_{i,1}\; , \qquad \Le_{i,1} = 
\begin{pmatrix}
 0 & \xi_i^1&\xi_i^2\\ 
 \xi_i^2&0&0\\
 \xi_i^1& 0&0\end{pmatrix}.
\]
For a single step of the matrix random walk we have
\[
U_{i,1}^* \Le_{i,1} = \big( \xi_i^2 v_i^1 + \xi_i^1 v_i^2\, , \; \xi_i^1 v_i^0\,, \; \xi_i^2 v_i^0 \big).
\]
Expanding $\det(W)$ using multilinearity of the determinant, we have
\begin{align*}
p(\xi^1,\xi^2) = \det(W) 
&= \sum_{i_0,i_1,i_2\in [N]}\left\lbrace \xi_{i_0}^2 \xi_{i_1}^1\xi_{i_2}^2 \det( v_{i_0}^1, v_{i_1}^0, v_{i_2}^0) 
+ \xi_{i_0}^1 \xi_{i_1}^1\xi_{i_2}^2 \det( v_{i_0}^2, v_{i_1}^0, v_{i_2}^0) \right\rbrace
\end{align*}
where we view the determinant as a degree-3 polynomial $p$ in the i.i.d.\ complex Gaussian variables $\xi^1=(\xi_i^1)_{i\in [N]}, \xi^2 = (\xi_i^2)_{i\in [N]}$.
Now a simple consequence of the Carbery--Wright inequality (cf.\ \Cref{lem:anti-poly}) gives an anticoncentration bound of the form
\begin{equation}	\label{outline:unstruct}
\pr\{ |p(\xi^1,\xi^2)| \le \eps \} \ls N^{O(1)} \eps^{1/3}
\end{equation}
as soon as we can find a monomial of $p$ with coefficient of size at least $N^{-O(1)}$. 
It will be sufficient to focus on the coefficients of the $2N^2$ monomials $\xi_i^1(\xi_j^2)^2$ and $(\xi_i^1)^2 \xi_j^2$ for $1\le i,j\le N$, which are
\begin{equation}	\label{outline:coeff}
\Delta_{i,j}^1:= \det( v_j^1,v_i^0,v_j^0)\,, \qquad \Delta_{i,j}^2 := \det( v_i^2,v_i^0,v_j^0).
\end{equation}

\subsection{Step 4: Ruling out structured bases}

We hence obtain \eqref{outline:goal3} as soon as either $|\Delta_{i,j}^1|=N^{-O(1)}$ or $|\Delta_{i,j}^2|=N^{-O(1)}$ for some $i,j\in [N]$. 
We thus say that the set $\UU_1$ of orthonormal columns in $(\subsV^z_{(1)})^\perp$ is \emph{structured} if 
\begin{equation}	\label{outline:struct}
|\Delta_{i,j}^1|, \,|\Delta_{i,j}^2| \le N^{-\gamma} \qquad \forall\; i,j\in [N]
\end{equation}
for some $\gamma>0$ to be taken sufficiently large, and aim to show that, in the randomness of the remaining columns $\col_j(\LLz)$, $2\le j\le N$, $\UU_1$ is unstructured except with probability $1-O(e^{-N})$. 

Previous works on the invertibility of i.i.d.\ matrices have controlled the event that a set of columns has a structured normal vector using net arguments, and we do the same here. 
However, the notion of structure here is quite different from the arithmetic structure encountered in those works, as it involves an orthonormal set rather than a single normal vector, and involves relations between a set of tuples (in this case triples) of coordinates that is determined by the polynomial $\pol$. Indeed, one notes that the pattern of indices in \eqref{outline:coeff} results from the form of the linearization \eqref{outline:LLz}, which will be different for other polynomials (and in fact there are multiple linearizations one can consider for any given polynomial). 

To bound the probability that $\UU_1$ satisfies \eqref{outline:struct} we aim to construct an $\eps$-net $\cN$ for the set of all possible structured orthonormal bases $\UU=(U_1,\dots, U_N) \in \Mat_3(\C)^N\cong \Mat_{3N,3}(\C)$. Since $\UU_1^*\col_j(\LLz) = 0_{3\times3}$ for every $2\le j\le N$, then if $\UU_1$ is approximated within $\eps$ (in the Hilbert--Schmidt metric) by an element $\hUU$ of the net, we have $\|\hUU^*\col_j(\LLz)\|_\HS = O(\eps)$ for all $2\le j\le N$ (here we use that $\LLz$ has bounded operator norm with high probability).  Viewing the $3\times3$ random matrix $\hUU^*\col_j(\LLz)$ as a random vector in $\C^9$, after a bit of algebra one sees 
\[
\|\hUU^*\col_j(\LLz)\|_\HS = \| \Flat^* \bxi\|_2
\]
where $\bxi= (X_1(i,1),\dots, X_1(i,N), X_2(i,1),\dots, X_2(i,N))\in \C^{2N}$ is a Gaussian vector and
\begin{equation}	\label{outline:WW}
\Flat = \begin{pmatrix} \hU^1 & \hU^0& 0\\ \hU^2 & 0 & \hU^0 \end{pmatrix}  \in \Mat_{2N, 9}(\C)
\end{equation}
where $\hU^0,\hU^1,\hU^2\in \Mat_{N,3}(\C)$ denote the blocks of $\hUU$ under the partitioning of coordinates as in \eqref{outline:LLz}. 
Thus, the small ball probability for the matrix random walk $\hUU^*\col_j(\LLz)\in \Mat_3(\C)$ is controlled by that of the Gaussian vector $\Flat^*\bxi\in \C^9$, which in turn is determined by the effective rank $r$ of $\Flat$ (one can take $r$ to be the number of singular values of $\Flat$ that exceed some small fixed threshold). 
 
Note that the conditions \eqref{outline:struct} are geometric in nature, as they imply that a large collection of triples of the rows are not in general position, but lie in some hyperplane in $\C^3$. This is what allows for the construction of an efficient net for the set $\Struct$ of structured bases $\UU$.  
Supposing we can find a net of size $|\cN| = O(1/\eps^2)^{dN}$, where $0\le d\le 9$ is the effective metric dimension of $\Struct$ (which is a subset of the $9N$-dimensional complex vector space $\Mat_{3N,3}(\C)$, so that this estimate holds trivially for $d=9$) we can then bound the probability that $\UU_1$ is structured by
\[
\sum_{\hUU\in \cN} \pr\{ \|\hUU^*\col_j(\LLz)\|_\HS\le C\eps\} \le O(1/\eps^2)^{dN} \cdot O(\eps^2)^{r(N-1)} = O(\eps^2)^{2(r-d)N - 2r}.
\]
We thus hope to show that the effective rank $r$ of matrices $\Flat$ as in \eqref{outline:WW} is strictly larger than the effective metric dimension $d$ of $\Struct$. 
Note that the former is determined by relations between the columns of $\hUU$, while the latter is determined by relations between the rows. 

In fact one cannot show the effective rank $r$ is strictly larger than the effective dimension $d$ uniformly for all structured $\UU$ -- we will instead need to stratify the set $\Struct$ according to the effective rank and use the geometric relations \eqref{outline:struct} between rows together with the rank constraint to construct an efficient net on each stratum. 
We defer further explanation of this step to \Cref{sec:struct}.

\section{Notation and preliminaries}
\label{sec:notation}

\subsection{Asymptotic notation}

We use $C,C_0, c,c',$ etc.\ to denote positive, finite constants. If the constants depend on parameters $p,q,\dots$ we indicate this by writing $C=C(p,q,\dots)$. If no dependence on parameters is given then the constants are understood to be absolute.

For $g\in \R_+$ we write $O(g)$ to stand for a quantity  $f\in \C$ such that $|f|\le Cg$ for some absolute constant $C\in (0,+\infty)$.
For $f,g\in \R_+$, $f\ls g$ and $g\gs f$ mean that $f=O(g)$.
For a parameter (or list of parameters) $q$, we write $f=O_{q}(g)$ to mean $|f|\le C g$ for some $C=C(q)\in (0,+\infty)$, and similarly for $f\ls_qg,f\gs_q g$.

\subsection{Matrices} 
\label{sec:matrices}

For integers $n\ge m\ge1$ we use the common abbreviations
$[m,n]:= \{m,\dots, n\}$ and $[n]:=[1,n]$.  
We write $|I|$ for the cardinality of a finite set $I$,
and for $n<|I|$ we denote by $I\choose n$ the set of subsets of $I$ with cardinality $n$.

We write $\Mat_{n,m}(\cX)$ for the set of $n\times m$ matrices with entries in a set $\cX$, and abbreviate $\Mat_n(\cX):= \Mat_{n,n}(\cX)$. 
The $n\times n$ identity matrix is denoted $\id_n$, with the subscript sometimes omitted if it is clear from the context.
For $a\in \C$ we often write $a$ for $a\id_n$ when there can be no confusion.

We frequently work with block matrices, and the following more general matrix notation facilitates referral to their entries and submatrices of various dimensions. For finite indexing sets $S,T$ we write $\matt{S}{T}{\cX}$ for the set of $|S|\times |T|$ matrices with entries in $\cX$ and rows and columns indexed by $S$ and $T$, respectively. 
Thus, $\Mat_{n,m}(\cX) = \Mat_{[n]}^{[m]}(\cX)$. 
We write $\Mat_S(\cX):= \Mat_S^S(\cX)$ when there can be no confusion (we avoid this when $S$ is itself a product set).
For $A\in\matt{S}{T}{\cX}$ and $(s,t)\in S\times T$ we write $A(s,t)$ for the $(s,t)$ entry of $A$. 
A block matrix $\bAA$ that is an $n\times n'$ array of blocks $A^{k,\ell}\in \Mat_{N,N'}(\cX)$ is viewed as an element of 
\[
\matt{[n]\times[N]}{[n']\times[N']}{\cX} 
\]
and is naturally associated with an element of $\Mat_{n,n'}(\Mat_{N,N'}(\cX))$; we abusively write $\bAA$ for both. 
Thus,
\begin{equation}	\label{not:super}
\bAA(k,\ell) = A^{k,\ell}, \qquad 
\bAA((k,i),(\ell,j)) = A^{k,\ell}(i,j)
\end{equation}
for $(k,\ell)\in [n]\times[n']$ and $(i,j)\in [N]\times[N']$. 
We will sometimes index the rows by $I\times J$ for some other sets $I,J\subset \Z_{\ge 0}$ of size $n, N$, respectively, and similarly for the columns. 
It will often be convenient to view a block matrix ``inside-out'' -- that is, $\bAA$ is naturally associated to an element
\begin{equation}	\label{def:tmat}
\tAA\in \matt{[N]\times [n]}{[N']\times[n']}{\cX} \cong \Mat_{N,N'}(\Mat_{n,n'}(\cX))
\end{equation}
with $\tAA((i,k),(j,\ell)) = \bAA((k,i),(\ell,j))$.

Typically, one of the pairs $(N,N')$ will contain a ``large'' dimension (usually $(N,N')\in \{(N,N),(N,N-1),(N,1)\}$, with $N$ the size of the Ginibre matrices in Theorems \ref{thm:brown} and \ref{thm:pseudo}) while the other pair $(n,n')$ remains bounded (being related to parameters of the fixed polynomial $\pol$).
In this case we always use boldface for the block matrix $\bAA$, with its large submatrices $A^{k,\ell}$ indexed with superscripts and its small submatrices indexed with subscripts:
\begin{equation}
\tAA(i,j) = A_{i,j} \, , \qquad (i,j)\in [N]\times[N'].
\end{equation}
Thus
\begin{equation}	\label{inside-out}
\bAA((k,i),(\ell,j)) = A^{k,\ell}(i,j) = A_{i,j}(k,\ell) = \tAA((i,k),(j,\ell)).
\end{equation}
We always use indices $i,j,i_0,i_1,$ etc.\ for the large dimensions and $k,\ell,$ etc.\ for the small dimensions.

For $A\in \matt{S}{T}{\cX}$, we denote its rows and columns 
\[
\row_s(A) = (A(s,t))_{t\in T} \in \cX^{T}, \qquad \col_t(A) = (A(s,t))_{s\in S} \in \cX^S.
\]
We sometimes manipulate these as row and column vectors (note that the entries may be matrices). 
Thus, for $\bAA = (A^{k,\ell})\in \Mat_{n,n'}(\Mat_{N,N'}(\C))$ with $\tAA = (A_{i,j})\in \Mat_{N,N'}(\Mat_{n,n'}(\C))$ we have that
\begin{align*}
\col_\ell(\bAA) &= (A^{1,\ell},\dots, A^{n,\ell})\in \Mat_{N,N'}(\C)^n, \qquad \ell\in [n']\\
\col_j(\tAA) &= (A_{1,j},\dots, A_{N,j})\in \Mat_{n,n'}(\C)^N, \qquad \ell\in [N']
\end{align*}
are vectors of matrices,
whereas
\begin{align*}
\col_{(\ell,j)}(\bAA) & = (\bAA((k,i),(\ell,j)))_{(k,i)\in [n]\times[N]}\in \C^{[n]\times[N]}, \qquad (\ell,j)\in [n']\times [N']\\
\col_{(j,\ell)}(\tAA) & = (\bAA((k,i),(\ell,j)))_{(i,k)\in [N]\times[n]}\in \C^{[N]\times[n]}, \qquad (j,\ell)\in [N']\times [n']
\end{align*}
are arrays of scalars, which we view as vectors with block coordinate structure.

For $A\in \matt{S}{T}{\cA}$ and $B\in \matt{T}{U}{\cA}$ with $\cA$ a $*$-algebra we define the product
\[
AB = \big(\row_s(A)\cdot \col_u(B)\big)_{s\in S,u\in U} = \Big( \sum_{t\in T} A(s,t)B(t,u) \Big)_{s\in S,u\in U} \in \matt{S}{U}{\cA}
\]
and the conjugate transpose
\[
A^* = \big(A(t,s)^*\big)_{t\in T,s\in S} \in \matt{T}{S}{\cA}
\]
in the usual way.
For $A\in \matt{S}{T}{\cA}$ and $B\in \matt{U}{V}{\cA}$ we denote the tensor product $A \otimes B\in \matt{S\times U}{T\times V}{\cA}$ with entries
\[ 
(A \otimes B)((s,u),(t,v)) = 
A(s,t)B(u,v).
\]

\subsection{Norms and singular values}

For finite indexing sets $S,T$ we equip $\C^T$ with the Euclidean inner product $\langle\cdot,\cdot\rangle$ and $\ell_2$-norm $\|\cdot\|_2$,
and $\matt{S}{T}{\C}$ with the Hilbert--Schmidt norm
\[
\|A\|_\HS = \Big( \sum_{s\in S, t\in T} |A(s,t)|^2 \Big)^{1/2}.
\]
We let $\ball^{T}\subset \C^T$ and $\ball_S^T\subset \matt{S}{T}{\C}$ denote the closed unit balls under the Euclidean $\ell_2$ and Hilbert--Schmidt norms, respectively.
We denote the boundary of $\ball^T$ by $\sph^T$. For $T=[n]$ we write $\ball^n$, $\sph^{n-1}$. 
We sometimes write $\dist(u,v):= \|u - v\|_2$ for the Euclidean metric on $\C^T$.
We also equip $\matt{S}{T}{\C}$ with the $\ell_2(T)\to \ell_2(S)$ operator norm
\[
\|A\|_\op = \sup_{v\in \sph^T} \| Av\|_2\,.
\]
For $A\in \matt{S}{T}{\C}$ we label its singular values in non-increasing order:
\[
\|A\|_{\op}=\sigma_1(M)\ge \cdots\ge \sigma_{|S|\vee|T|}(M) \ge 0\,,
\]
where the last $|S|\vee|T|-|S|\wedge|T|$ singular values are trivially zero, and denote by
\[
\sigma_{\min}(A):= \sigma_{|S|\wedge|T|}(M)\ge0
\]
the smallest nontrivial singular value (which may be zero). 
In particular, for $|S|=|T|$ we have that $A$ is invertible if and only if $\sigma_{\min}(A)>0$. 
For a square block matrix $\bAA\in \matt{[n]\times[N]}{[n]\times [N]}{\C}$ we have $\sigma_l(\bAA) = \sigma_l(\tAA)$ for all $1\le l\le nN$, as the singular values are invariant under relabeling of the rows and columns.

For subsets $E_1,\dots, E_m$ of a vectors space we write $\Span(E_1,\dots, E_m):= \Span(E_1\cup \cdots E_m)$ for the linear span of their union. With slight abuse we allow some of the $E_k$ to be single points, understood to denote the singleton sets $\{E_k\}$. 
For $A\in \matt{S}{T}{\C}$, by $\Span(A)$ we mean the linear span of its columns in $\C^S$. 
For vectors $v_j\in \C^S$, $j\in J$ we sometimes write  $\langle v_j: j\in J \rangle = \Span(\{ v_j\}_{j\in J})$ and $\langle v_j\rangle = \Span(v_j)$. $v_{j}\wedge v_{k}$ denotes the wedge product of $v_{j}$ and $v_{k}$.

\subsection{Basic facts}
\label{sec:facts}

Recall that an \emph{$\eps$-net} for a subset $E$ of a metric space $(\cX,d)$ is a finite set $\Sigma\subset E$ such that 
\[
\sup_{x\in E} \min_{y\in \Sigma} d(x,y) <\eps.
\]
In the sequel all nets will be with respect to the appropriate Euclidean metric (which for matrices coincides with the Hilbert--Schmidt norm).
The following is standard: 

\begin{lemma}\label{lem:net}
There is an absolute constant $C>0$ such that for 
any subset $E$ of the closed ball of radius $R$ in $\C^d$ with Euclidean metric
and any $\eps\in (0,R)$,
there is an $\eps$-net for $E$ of cardinality at most $(CR/\eps)^{2d}$.
\end{lemma}

The following two lemmas provide control on the norm and the norm of the inverse of the individual Ginibre inputs $X_1,\dots, X_\num$ for the polynomial $\Pol$.

\begin{lemma}
\label{lem:norm}
There is a constant $C_0<\infty$ such that the following holds. 
For any $N\ge 1$ and $X$ an $N\times N$ Ginibre matrix, 
\[
\pr\{ \|X\|_\op>C_0\} \ls e^{-N}.
\]
\end{lemma}

\begin{proof}
This holds more generally for matrices with independent complex uniformly-sub-Gaussian entries; see for instance \cite{RuVe:LO} (one reduces to the real Ginibre case with the triangle inequality).
\end{proof}

\begin{lemma}
\label{lem:smin}
Let $N\ge1$ and $X$ an $N\times N$ Ginibre matrix. For any  deterministic $M\in \Mat_N(\C)$ and any $\eps>0$
\[
\pr\{ \smin(X+M) \le \eps \} \ls N^3\eps^2 \,,
\]
where we emphasize that (following our previously stated convention) the implied constant is absolute and in particular is independent of the shift $M$.
\end{lemma}

We include the short proof below for completeness.
We remark in passing that Edelman obtained the sharp bound of $N^2\eps^2$ (with no constant factor loss) on the left hand side for the unshifted case $M=0$ \cite{Edelman:condition}. The sharpening to $O(N^2\eps^2)$ for the shifted case can be found in \cite{BKMS} (cf.\ Lemma 3.3 there), extending the analogous bound of $O(N\eps)$ for real shifts of real Ginibre matrices from \cite{SST}.

\begin{proof}
Write $A= X+M$. On the event that $\smin(A)\le \eps$, there exists $u\in \sph^{N-1}$ such that $\|Au\|_2\le \eps$. Moreover, since $u$ must have a coordinate $i$ such that $|u_i|\ge 1/\sqrt{N}$, we obtain after projecting the vector $Au$ to the orthocomplement of $A_{-i}:=\Span \{\col_j(A): j\ne i\}$ that 
\[
\dist( \col_i(A) , A_{-i}) \le \eps\sqrt{N}.
\]
By rotational invariance of the distribution of $\col_i(X)$, the probability of the above event is equal to 
\[
\pr\{ \dist(\col_i(X) + y, \Span\{e_2,\dots, e_N\}) \le \eps\sqrt{N} \}
\]
for a vector $y$ that depends only on $M$ and $A_{-i}$, and hence is independent of $\col_i(X)$.
Conditioning on the columns $\col_j(X)$ with $j\ne i$, the above is equal to
\[
\pr\{ |X(1,i) + y(1) | \le \eps\sqrt{N}\} \ls \eps^2N^2,
\]
where we used that $\sqrt{N}\cdot X(1,i)\sim \Normal_\C(0,1)$ has bounded density on $\C$. 
Taking a union bound over the possible choices of $i$ yields the claim.
\end{proof}

\section{Linearization}
\label{sec:linearization}

Recall that $\C\langle\xx\rangle=\C\langle x_1,\dots, x_n\rangle$ is the set of polynomials with complex coefficients in $n$ non-commuting indeterminates $x_1,\dots, x_n$. 
We express $\pol\in \C\langle\xx\rangle$ of degree two as
\begin{equation}	\label{pol.express}
\pol(\xx) = \sum_{\ell,m=1}^\num a_{\ell, m} x_\ell x_m + \sum_{\ell=1}^\num b_\ell x_\ell + \shift.
\end{equation}

\begin{lemma}[Linearization]
\label{lem:linearize}
Let $\cA$ be a unital $*$-algebra over $\C$, 
let $x_1,\dots, x_\num\in \cA$, and let
$\pol \in \C\langle \xx \rangle$ have degree two. 
Let $\rnk\in [\num]$ be the rank of the matrix $A=(a_{\ell,m})$ as in \eqref{pol.express}.
There exist vectors
$
\rlin_1,\dots,\rlin_\rnk, \slin_0, \slin_1,\dots, \slin_\rnk \in \C^\num
$
such that $\rlin_1,\dots, \rlin_{\rnk}$ are orthonormal and $\slin_1,\dots,\slin_{\rnk}$ are nonzero and orthogonal (note we leave out $\slin_0$), and the following holds.
Define the linear mapping
\begin{equation}	\label{def:lin}
\Lin=\Lin_\pol: \cA^\num \to \Mat_{[0,\rnk]}(\cA), \qquad 
\Lin(\xx):= 
\begin{pmatrix}
\langle \slin_0, \bx\rangle & \langle\rlin_1,\bx\rangle & \cdots & \langle\rlin_\rnk,\bx\rangle \\
\langle\slin_1,\bx\rangle & 0 & \cdots & 0 \\
\vdots & \vdots & \ddots & \vdots \\
\langle\slin_\rnk,\bx\rangle & 0 &\cdots & 0 \\
\end{pmatrix}
\end{equation}
and for $z\in \C$ set
\begin{equation}	\label{def:linz}
\Linz=\Linz_\pol: \cA^\num \to \Mat_{[0,\rnk]}(\cA), \qquad 
\Linz(\xx):= \Lin(\xx) +
\Kez
\otimes 1_\cA
\end{equation}
where 
\begin{equation}	\label{def:shift}
\Kez=\Kez_\pol:=
\begin{pmatrix}
\shift-z & 0 \\ 
0 & -\id_\rnk
\end{pmatrix} 
\in \Mat_{\DD}(\C)\,.
\end{equation}
We have that
for any $z\in\C$,
\begin{equation}
(\pol-z)^{-1} = ((\cL^z(\xx))^{-1})_{1,1}.
\end{equation}
\end{lemma}

For the proof we recall the Schur complement formula: for $M\in \Mat_{m+n}(\cA)$ written in block form as
\[
M= \begin{pmatrix} M_{1,1} & M_{1,2} \\ M_{2,1} & M_{2,2} \end{pmatrix}, \quad M_{1,1} \in \Mat_m(\cA), \; M_{2,2}\in \Mat_n(\cA)
\]
with $M_{2,2}$ invertible, we have that the top-left $m\times m$ block of $M^{-1}$ is given by
\[
(M^{-1})_{1,1} = \big( M_{1,1} - M_{1,2} (M_{2,2})^{-1} M_{2,1} \big)^{-1}.
\]

\begin{proof}
For $\slin_0= (\slin_{0,1},\dots, \slin_{0,\num})$ we take $\slin_{0,\ell} = \overline{b_\ell}$. 
Write $\pol_{(2)}(\bx) = \sum_{\ell, m=1}^\num a_{\ell,m} x_\ell x_m$ for the homogeneous degree-2 part of $\pol(\bx)$.
Let $A= U\Sigma V^*$ be the singular value decomposition for $A$, with $U=(u_{\ell, k}), V= (v_{m,k})\in \Mat_{\num,\rnk}(\C)$, and $\Sigma =\diag (\sigma_1,\dots, \sigma_\rnk)$ with $\sigma_1\ge\dots\ge \sigma_\rnk>0$.
For $1\le k\le \rnk$ and $1\le \ell\le \num$, we set $\rlin_{k,\ell} := \overline{u_{\ell,k}}$ and $\slin_{k,\ell} = \sigma_kv_{\ell, k}$, so that
\begin{align*}
\pol_{(2)}(\bx) 
&= \sum_{k=1}^\rnk \sigma_k \sum_{\ell=1}^\num u_{\ell, k} x_\ell \sum_{m=1}^\num \overline{v_{m, k}} x_m\\
&= \sum_{k=1}^\rnk \bigg( \sum_{\ell=1}^\num \overline{\rlin_{k,\ell}} x_\ell \bigg) \bigg( \sum_{m=1}^\num \overline{\slin_{k,m}} x_m \bigg)\\
& = \sum_{k=1}^\num \langle r_k, \bx\rangle \langle s_k,\bx\rangle.
\end{align*}
With $\Linz(\bx)$ as in \eqref{def:linz}, by the Schur complement formula we have
\[
((\Linz(\bx))^{-1})_{1,1} = \Big(\, \langle \slin_0,\bx\rangle +\shift -z + \sum_{k=1}^\rnk \langle \rlin_k,\bx\rangle \langle \slin_k,\bx\rangle \,\Big)^{-1} = (\pol(\bx) - z)^{-1}
\]
as desired.
\end{proof}

\begin{cor}
\label{cor:linearize}
Let $\XX=(X_1,\dots, X_\num)$, and $\Pol$ be as in \Cref{thm:pseudo}, with $\pol$ as in \eqref{pol.express}. With $\rnk$ the rank of $A=(a_{\ell,m})$, there exist vectors $\slin_0,\slin_1,\dots, \slin_\rnk \in \C^\num$ depending only on $\pol$, 
with $\slin_1,\dots,\slin_\rnk$ nonzero and mutually orthogonal, such that the following holds. 
Let
\begin{equation}	\label{def:LLz}
\LLz :=
\LL+ \Kez\otimes \id_N
\end{equation}
with $\Kez$ as in \eqref{def:shift} and
\begin{equation}
\LL := 
\begin{pmatrix} 
Y_0 & X_1 & \cdots & X_\rnk \\ 
Y_1 & 0 & \cdots & 0\\ 
\vdots & \vdots & \ddots & \vdots \\
Y_\rnk & 0 & \cdots & 0 
\end{pmatrix}\,,
\end{equation}
where
\begin{equation}
Y_k := \sum_{\ell=1}^\num \overline{\slin_{k,\ell}} X_\ell, \qquad k\in [0,\rnk].
\end{equation}
For all $\eps>0$ and $z\in \C$, 
\begin{equation}	\label{bd.lin}
\P\{ \sigma_{\min}(P-z)\le \eps\} \le \P\{ \sigma_{\min}(\LLz)\le \eps\} .
\end{equation}
\end{cor}

\begin{proof}
We apply \Cref{lem:linearize} to get a linearized matrix $\Linz(\XX)$ as in \eqref{def:linz}, with
\begin{equation*}	
(\Pol-z)^{-1} = ((\Linz(\XX))^{-1})_{1,1}
\end{equation*}
where the right hand side is the $N\times N$ top-left block of $(\Linz(\XX))^{-1}$. In particular,
\begin{equation}\label{corlin1}
\|(\Pol-z)^{-1}\|_\op = \|((\Linz(\XX))^{-1})_{1,1}\|_\op \le \|(\Linz(\XX))^{-1}\|_\op.
\end{equation}
Now extend $\rlin_1,\dots, \rlin_\rnk$ to an orthonormal basis $\C^n$ and let $R$ be the unitary matrix with rows $\rlin_1,\dots, \rlin_\num$. Replacing the $N^2$ independent complex standard Gaussian vectors $\XX(i,j)=(X_1(i,j),\dots, X_\num(i,j))$ with $R^*\XX(i,j)$ and $\slin_0,\dots, \slin_\rnk$ with $\slin_0R^*, \dots, \slin_\rnk R^*$ we obtain the matrix \eqref{def:LLz}. 
The claim now follows from \eqref{corlin1} and the invariance of the complex standard Gaussian measure on $\C^\num$ under unitary transformations.
\end{proof}

It will be convenient later in the proof to assume $\rnk\ge 2$, so we now dispense with the case that $\rnk=1$. 

\begin{lemma}
\label{lem:rnk1}
\Cref{thm:pseudo} holds for the case that $A=(a_{\ell,m})$ as in \eqref{pol.express} has rank one. 
\end{lemma}

\begin{proof}
With a unitary change of basis we can reduce to a matrix of the form 
\[
P =  X_1(\alpha_1X_1+\alpha_2X_2) + \sum_{k=1}^3 \beta_kX_k + \gamma .
\]
If $\beta_3\ne0$ then, conditioning on $X_1,X_2$, we have from \Cref{lem:smin} that 
\[
\pr\{ \smin(P)\le \eps\} = \pr\{ \smin(X_3 + M) \le \eps/|\beta_3|\} \ls (\eps/|\beta_3|)^2N^3 = O_\pol(N^3\eps^2),
\]
where $M$ is a deterministic shift depending on $X_1,X_2$.
Having obtained the claim in this case, we henceforth assume that $\beta_3=0$.
Now write
\[
P = (\alpha_2 X_1 + \beta_2  ) X_2 + (\alpha_1X_1^2 + \beta_1X_1 + \gamma).
\]
If $\alpha_2$ or $\beta_2$ is nonzero then, conditioning on $X_1$, we have
\begin{align*}
\pr\{ \smin(P)\le \eps \} 
&\le \pr \{ \smin(\alpha_2X_1+\beta_2)\le \delta\} 
+ \pr\{ \smin( X_2 + M' ) \le \eps/\delta\}
\end{align*}
where $M' = (\alpha_2X_1 + \beta_2)^{-1} (\alpha_1X_1^2 + \beta_1X_1 + \gamma)$ is a deterministic shift depending on $X_1$ (note that $M'$ is well defined off a null event since $\alpha_2X_1+\beta_2$ is almost-surely invertible). 
Taking $\delta = \min\{ \sqrt{\eps}, |\beta_2|/2\}$, the claim again follows from \Cref{lem:smin}. 

Finally, assuming $\beta_3=\beta_2=\alpha_2=0$, we have
\[
P = \alpha_1X_1^2 + \beta_1X_1 + \gamma .
\]
If $\alpha_1=0$ we can conclude along the same lines as in previous cases. Otherwise we can factorize 
$
P = \alpha_1 (X_1 - a_+)(X_1-a_-) 
$
for some $a_\pm\in \C$.
Then since $\smin((X_1 - a_+)(X_1-a_-)) \ge \smin(X_1-a_+)\smin(X_1-a_-)$ we have
\begin{align*}
\pr\{\smin(P)\le \eps\} 
&\le \pr\{ \smin(X_1-a_+) \le (\eps/|\alpha_1|)^{1/2}\} + \pr\{ \smin(X_1-a_-) \le (\eps/|\alpha_1|)^{1/2}\}
\ls_\pol \eps N^3
\end{align*}
and the lemma is proved.
\end{proof}

\section{Reduction to a bounded-dimensional test projection}
\label{sec:bdd}

In view of \Cref{cor:linearize}, we henceforth write
\begin{equation}	\label{def:Linz2}
\Linz(\xx) = \Lin(\xx) + \Kez\otimes 1_\cA, \quad\text{with}\quad
\Lin(\xx):= 
\begin{pmatrix}
\langle \slin_0, \bx\rangle &x_1 & \cdots & x_\rnk \\
\langle\slin_1,\bx\rangle & 0 & \cdots & 0 \\
\vdots & \vdots & \ddots & \vdots \\
\langle\slin_\rnk,\bx\rangle & 0 &\cdots & 0 \\
\end{pmatrix}
\end{equation}
(i.e. taking $r_k$ to be the $k$-th standard basis vector in $\C^n$ for each $k\in [\rnk]$), so that $\LL= \Lin(\XX), \LLz= \Linz(\XX)$ are as in \Cref{cor:linearize}.
Our aim is to establish a lower tail bound for $\sigma_{\min}(\LLz)$. 

Recall our notational conventions from \Cref{sec:matrices}.
We view $\LL,\LLz$ as elements of $\matt{\DD\times[N]}{\DD\times[N]}{\C}\cong\Mat_{\DD}(\Mat_N(\C))$ (note that the first row and column have index $0$). 
Thus, $\LL$ has entries
\[
\LL((k,i),(l,j)) = \Le^{k,l}(i,j) = \Le_{i,j}(k,l) = \tLL((i,k),(j,l)), \quad k,l\in \DD,\; i,j\in [N]
\]
with
\[
\Le_{i,j} = \Lin (X_1(i,j), \dots, X_n(i,j))\in \Mat_\DD(\C),
\]
and similarly $\LLz((k,i),(l,j)) = \Lez_{i,j}(k,l)$, with
\[
\Lez_{i,j} = \Linz(X_1(i,j),\dots, X_\num(i,j)) = \Le_{i,j} + \Kez\delta_{i,j}.
\]
Note that as an element of $\Mat_N(\Mat_\DD(\C))$, 
$
\tLLz= (\Lez_{i,j})_{i,j\in [N]}
$
has independent entries, and  $\tLL = (\Le_{i,j})_{i,j\in [N]}$ has entries that are i.i.d.\ and centered:
\[
\Le_{i,j} = \Lez_{i,j} - \E \Lez_{i,j} . 
\]
We stress that the $(\dd)^2$ entries of $\Le_{i,j}$ in general are \emph{not} independent.

For $1\le j\le N$ we write $\LLz_{(j)}\in \Mat_{N,N-1}(\Mat_{[0,\rnk]}(\C))$ (resp.\ $\LL_{(j)}$) for the submatrix of $\LLz$ (resp.\ $\LL$) obtained by removing the $j$th column from each $N\times N$ submatrix. We let $\subsV^z_{(j)} \subset \C^{\DD\times[N]}$ denote the span of the columns of $\LLz_{(j)}$, viewed as an element of $\matt{\DD\times [N]}{\DD\times[N-1]}{\C}$, that is
\begin{equation}\label{Vj}
\subsV^z_{(j)}:= \Span \{ \col_{(l,j')}(\LLz) : j'\in [N]\setminus \{j\},\, l\in \DD\}.
\end{equation}
For each $j$, conditional on these $(N-1)(\dd)$ columns, we draw a matrix
$\UU_j \in \matt{\DD\times[N]}{\DD}{\C}$ with columns that are orthonormal in $(\subsV^z_{(j)})^\perp$, with $\UU_j$ independent of $\col_j(\tLLz) = (\Lez_{i,j})_{i=1}^N$. 
(To be more precise, we can for instance draw $N$ i.i.d.\ Haar unitaries $H_j\in \cU(\dd)$, independent of $(X_1,\dots, X_\num)$, and fixing arbitrary matrices $\bs{V}_j$ with $\dd$ orthonormal columns in $\subsV^z_{(j)}$, chosen measurably with respect to the sigma algebra generated by $\{\Le_{i,j'}\}_{i\in [N], j' \in [N]\setminus \{j\}}$, we set $\UU_j = \bs{V}_jH_j$.)
Recalling our notational conventions from \Cref{sec:matrices}, we have that $\UU_j$ is naturally associated to $\tUU_j\in \Mat_{[0,\rnk]}(\C)^N$, an $N\times 1$ matrix of $(\rnk+1)\times (\rnk+1)$ blocks, which we denote by $U_{1,j},\dots, U_{N,j}$.

In the remainder of this section we establish the following:

\begin{lemma}[Reduction to invertibility of a small test projection]
\label{lem:bdd}
With notation as above, for $j\in [N]$ denote
\[
\hLz_j := 
\tUU_j^* \col_j(\tLLz) = \sum_{i=1}^N U_{i,j}^*\Lez_{i,j} \in \Mat_\DD(\C).
\]
Then for any $\eps>0$,
\[
\P\{ \sigma_{\min}( \LLz) \le \eps\} \le N\,\P\big\{\, |\det( \hLz_1)| \le \eps(|z|+C_\pol)^{\rnk} \sqrt{N} \,\big\} + e^{-N}
\]
for some $C_\pol<\infty$ depending only on $\pol$. 
\end{lemma}

\begin{proof}
On the event that $\smin(\LLz)\le \eps$ we have that there exists a unit vector $\bv=(v_1,\dots, v_N)\in \C^{[N]\times[0,\rnk]}$ with each $v_j\in \C^{[0,\rnk]}$ such that
\[
\eps \ge \|\LLz\bv \|_2 = \Big\| \sum_{j=1}^N \col_j(\tLLz) v_j \Big\|_2. 
\]
Since $\bv$ is a unit vector we must have $\|v_{j_0}\|_2\ge 1/\sqrt{N}$ for some $j_0\in [N]$. Since the norm of $\LLz\bv$ can only decrease under projection to the column span of $\UU_{j_0}$, we have
\[
\eps \ge \big\| \tUU_{j_0}^* \col_{j}(\tLLz) v_j \big\|_2 = \| \hLz_{j_0} v_{j_0}\|_2 \ge \frac1{\sqrt{N}} \smin( \hLz_{j_0}).
\]
Now note that
\[
\|\hLz_{j_0}\|_\op \le \|\LLz\|_\op \le \|\LL\|_\op + |z|+O_\pol(1).
\]
From \Cref{lem:norm} and the triangle inequality,
\[
\pr\{ \|\LL\|_\op > \bound\}  \le e^{-N}
\]
for some $\bound = O_\pol(1)$. Thus, by the union bound we have for some $\bound'=O_\pol(1)$, 
\begin{align*}
\pr\{ \smin(\LLz)\le \eps\} 
&\le e^{-N} + \sum_{j=1}^N \pr\{ \smin(\hLz_j)\le \eps\sqrt{N},\; \|\hLz_j\|_\op \le \bound' + |z|\}\\
&\le e^{-N} + \sum_{j=1}^N \pr\{ |\det(\hLz_j)| \le (\bound'+|z|)^{\rnk} \eps \sqrt{N}\}.
\end{align*}
Since the distribution of $\tLLz= (\Lez_{i,j})_{i,j\in [N]}$ is invariant under simultaneous permutation of the $N$ row and $N$ column indices, the claim follows.
\end{proof}

\section{Anti-concentration for matrix random walks}
\label{sec:unstruct}

From \Cref{lem:bdd} we see that it suffices to control the lower tail for the determinant of the $(\dd)$-dimensional matrix $\hLz_1$, which we may alternatively express as follows:
\begin{equation}		\label{hLz1.walk}
\hLz_1 = \tUU_1^* \col_1(\tLLz) = \sum_{i=1}^N U_{i,1}^* \Lez_{i,1}  \in \Mat_\DD(\C).
\end{equation}
The advantage of this perspective is that for a fixed realization of $\UU_1$ (which we recall is independent of $\col_1(\tLLz)$), the summands $U_{i,1}^*\Lez_{i,1}$ are $N$ independent random matrices.
We can thus view the matrix \eqref{hLz1.walk} as a random walk in $\Mat_\DD(\C)$. 

In this section we consider an arbitrary \emph{fixed} (deterministic) matrix 
\[
\UU\in \matt{\DD\times[N]}{\DD}{\C} \cong \Mat_N(\C)^\DD
\]
with orthonormal columns in $\C^{\DD\times[N]}$, identified with the sequence of its block submatrices as $\UU=(U^0,\dots, U^\rnk)\in \Mat_N(\C)^\DD$ and with the sequence of its square submatrices denoted $\tUU = (U_1,\dots, U_N)\in\Mat_\DD(\C)^N$. 
Thus, denoting the rows of $U_i$ by $(v_i^k)^*$, $k\in \DD$, we have that $(v_i^k)^*$ is the $i$th row of $U^k$.

Let $(g_i^k)_{i\in [N], k\in [\num]}, (h_i^k)_{i\in [N], k\in [\num]} \in\C^{[N]\times[\num]}$ be independent arrays of $N\num$ i.i.d.\ standard real Gaussians and denote $g_i= (g_i^1,\dots, g_i^\num)$, $g^k = (g_i^k)_{i\in [N]}$, and similarly for $h_i, h^k$.
We set $\xi_i^k = (2N)^{-1/2}(g_i^k+\ii h_i^k)$, $\xi_i = (\xi_i^1,\dots, \xi_i^\num)$, $\xi^k=(2N)^{-1/2}(g^k+\ii h^k)$. 
Then $\hLz_1$ in \eqref{hLz1.walk} is identically distributed to the random matrix
\begin{equation}	\label{def:W}
W= M + \sum_{i=1}^N U_i^*\Le_i  
\end{equation}
where
\begin{equation}	\label{def:M0}
M= M(\UU,z,\shift) := U_1^*\Kez \in \Mat_\DD(\C)
\end{equation}
is deterministic
and 
\begin{equation}	\label{def:Li}
\Le_i := \Lin(\xi_i^1,\dots, \xi_i^\num) = 
\begin{pmatrix}	
\langle\slin_0,\xi_i\rangle &  \xi_i^1 & \cdots & \xi_i^\rnk \\
\langle\slin_1,\xi_i\rangle & 0 & \cdots & 0 \\
\vdots & \vdots & \ddots & \vdots \\
\langle \slin_\rnk,\xi_i\rangle & 0 &\cdots & 0 \\
\end{pmatrix}
\in \Mat_\DD(\C)
\end{equation}
are i.i.d.,
with
$\shift,\slin_0,\dots, \slin_\rnk$ as in \Cref{cor:linearize}.

From \Cref{lem:bdd} we see that in order to prove \Cref{thm:pseudo}, it suffices to prove an anti-concentration estimate for $\det(W)$, a degree-$(\dd)$ polynomial in the $N\num$ Gaussian variables $g_i^k$,
of the form
\begin{equation}	\label{ACdet.goal}
\P\{ |\det(W)|\le \eps \} \le N^C\eps^c.
\end{equation}
In this section we identify sufficient structural conditions on the matrix $\UU$ in order to have \eqref{ACdet.goal}.
In \Cref{sec:struct} we will show that such conditions hold with high probability for the matrix $\UU_1$.

For $1\le \ell\le \num$ let
\begin{equation}	\label{def:Qell}
Q^\ell = Q^\ell(\UU) =  \sum_{k=0}^{\rnk} \slin_{k,\ell} U^k \in \matt{[N]}{\DD}{\C}
\end{equation}
and for $i\in [N]$ denote the $i$th row of $Q^\ell$ by $(w_i^\ell)^*$; thus,
\begin{equation}	\label{def:ws}
w_i^\ell =  \sum_{k=0}^\rnk \overline{\slin_{k,\ell}}v_i^k.
\end{equation}
For $\ell\in [\num]$ and $i_0,i_1,\dots, i_{\rnk}\in [N]$ denote
\begin{equation}	\label{def:Delta}
\Delta^\ell_{i_0,i_1,\dots, i_\rnk} = \Delta^\ell_{i_0,i_1,\dots, i_\rnk}(\UU) := \det( w_{i_0}^\ell, v_{i_1}^0, \dots, v_{i_\rnk}^0 ). 
\end{equation}
For a parameter $\pstr>0$ we define sets of \emph{structured matrix-columns}:
\begin{align}
\Struct_1(\pstr) & := 
\bigcap_{i_0,i_1,\dots, i_\rnk\in [N]}\bigcap_{\ell\in [\rnk+1,\num] }\Big\{ \UU\in \matt{[0,\rnk]\times[N]}{[0,\rnk]}{\C} :
 |\Delta^\ell_{i_0,i_1,\dots, i_\rnk}|<\pstr 
 \Big\} \,,	\label{def:struct1} \\
\Struct_2(\pstr) & :=\bigcap_{ i_1,\dots, i_\rnk\in [N]}\bigcap_{\ell\in [\rnk]} \Big\{ \UU\in \matt{[0,\rnk]\times[N]}{[0,\rnk]}{\C} :
 |\Delta^\ell_{i_\ell,i_1,\dots, i_\rnk}|<\pstr	
  \Big\}\,,
\label{def:struct2} \\
\Struct(\pstr) &:= \Struct_1(\pstr)\cap \Struct_2(\pstr).	\label{def:struct}
\end{align}

\begin{lemma}[Anti-concentration for the determinant of a matrix random walk]
\label{lem:anti-unstruct}
Let $\UU=(U_1,\dots,U_N)\in \matt{[0,\rnk]\times [N]}{ [0,\rnk]}{\C}$ and suppose $\UU\notin \Struct(\pstr)$ for some $\pstr>0$.
Then for any $\eps>0$, 
\begin{equation}	\label{unstruct.goal}
\sup_{M\in \Mat_{\DD}(\C)} \P \bigg\{ \bigg| \det \bigg( M + \sum_{i=1}^N U_i^* \Le_i \bigg) \bigg| \le \eps \bigg\}
 \ls_\rnk
N^{ 1/2 }  (\eps/\delta)^{1/(\rnk+1)}.
\end{equation}
\end{lemma}

To establish the lemma we make use of the following easy corollary of the Carbery--Wright inequality for anti-concentration of polynomials in Gaussian variables.

\begin{lemma}
\label{lem:anti-poly}
Let $g_1,\cdots,g_n$ be i.i.d.\ standard gaussian variables and let $\scpol$ be a real-valued degree-$d$ polynomial. Let $a\ne0$ be the coefficient of one of the degree-$d$ monomials of $\scpol$.
Then for any $\eps>0$, 
\[
\sup_{t\in\mathbb R} \P\left( |\scpol(g_1,\ldots,g_n)-t|\le \eps\right) \,\ls_d\, (\eps/|a|)^{1/d}.
\]
\end{lemma}

\begin{proof}
By rescaling $\scpol$ and $\eps$ by $a$ we may assume $a=1$. 
From the Carbery--Wright inequality \cite{CW} we have
\[
\sup_{t\in\mathbb R} \P\left( |\scpol(g_1,\ldots,g_n)-t|\le \eps\sigma \right) \,\ls_d\, \eps^{1/d}
\]
where $\sigma^2 = \Var(\scpol(g_1,\ldots,g_n))$ is the variance of $p$ under the product Gaussian measure. It hence suffices to verify that
\[
\Var(\scpol(g_1,\ldots,g_n))\,\gs_d\, 1.
\]
One readily notes the above indeed holds by, for instance, expanding $p$ in the orthonormal basis of Hermite polynomials and noting our assumption implies that one of the coefficients among the highest degree terms in the expansion must be of size $\gs_d1$. The claim follows.
\end{proof}

\begin{proof}[Proof of \Cref{lem:anti-unstruct}]
Write
\begin{equation}
D = \det \hL, \qquad \hL := M+ \sum_{i=1}^N U_i^*L_i.
\end{equation}
Denoting the columns of $M$ by $y^0,\dots, y^\rnk$, we have
\begin{align}
D &= \det \left( y^0 + \sum_{i_0=1}^N \sum_{k=0}^\rnk \langle \slin_k, \xi_{i_0}\rangle v_{i_0}^k \, ,\, y^1 + \sum_{i_1=1}^N \xi_{i_1}^1 v_{i_1}^0, \dots, y^{\rnk} + \sum_{i_\rnk=1}^N \xi_{i_\rnk}^\rnk v_{i_\rnk}^0 \right)	\notag\\
& = \det \left( y^0 + \sum_{i_0=1}^N \sum_{\ell=1}^\num \xi_{i_0}^\ell w_{i_0}^\ell \, ,\, y^1 + \sum_{i_1=1}^N \xi_{i_1}^1 v_{i_1}^0, \dots, y^{\rnk} + \sum_{i_\rnk=1}^N \xi_{i_\rnk}^\rnk v_{i_\rnk}^0 \right) .\label{det.expand}
\end{align}
Now we condition on the variables $\bh=(h_i^k)_{i\in [N], k\in [\num]}$, the imaginary parts of $\sqrt{2N}{\bxi}$,  and view $D$ as a degree-$(\rnk+1)$ polynomial in the $\num N$ i.i.d.\ standard real Gaussian variables $\bg = (g_i^k)$, the real parts of $\sqrt{2N}{\bxi}$.
Write $\wh{D}$ for the homogeneous degree-$(\rnk+1)$ part of this polynomial. We have
\begin{align}
\wh{D} &= (2N)^{-(\rnk+1)/2} \sum_{\ell=1}^\num \sum_{i_0,\dots, i_\rnk=1}^N g_{i_0}^\ell g_{i_1}^1\cdots g_{i_\rnk}^\rnk \cdot \det(w_{i_0}^\ell, v_{i_1}^0,\dots, v_{i_\rnk}^0 ) \\
&= (2N)^{-(\rnk+1)/2} \sum_{\ell=1}^\num \sum_{i_0,\dots, i_\rnk=1}^N g_{i_0}^\ell g_{i_1}^1\cdots g_{i_\rnk}^\rnk \cdot \Delta^\ell_{i_0,\dots, i_\rnk}.
\end{align}
In particular, for $\ell\le \rnk$, the coefficient of the monomial $g_{i_\ell}^\ell g_{i_1}^1\cdots g_{i_\rnk}^\rnk$ is 
\begin{equation}	\label{coeffA}
2\cdot (2N)^{-(\rnk+1)/2} \Delta^\ell_{i_\ell, i_1,\dots, i_\rnk}
\end{equation}
and for $\ell>\rnk$, the coefficient of the monomial $g_{i_0}^\ell g_{i_1}^1\cdots g_{i_\rnk}^\rnk$ is
\begin{equation}	\label{coeffB}
(2N)^{-(\rnk+1)/2} \Delta^\ell_{i_0, i_1,\dots, i_\rnk}.
\end{equation}
Since $\UU\notin\Struct(\pstr)$, at least one of these coefficients has modulus at least $ (2N)^{-(\rnk+1)/2}\delta$.
Suppose this holds  for a coefficient \eqref{coeffA} for some $\ell\le \rnk$ and $i_1,\dots, i_\rnk\in [N]$. Then either the real or imaginary part is of size at least $2^{-1/2}(2N)^{-(\rnk+1)/2}\delta$. 
Supposing further that this holds for the real part, then for the real Gaussian polynomial $\Re D$ there is a coefficient of a maximal degree monomial of size at least $2^{-1/2}(2N)^{-(\rnk+1)/2}\delta$. Applying \Cref{lem:anti-poly} to $\Re D$, we have
\[
\pr\{|D|\le \eps\} \le \pr\{ |\Re D|\le \eps\} \ls_\rnk N^{1/2} (\eps/\pstr)^{1/(\rnk+1)}
\]
yielding \eqref{unstruct.goal} in this case. For the case that the imaginary part of this coefficient is of size at least $2^{-1/2}(2N)^{-(\rnk+1)/2}\delta$ we argue similarly with the real Gaussian polynomial $\Im D$,
and we repeat the same reasoning for the case that a coefficient as in \eqref{coeffB} is large.
\end{proof}

\section{Ruling out structured bases}
\label{sec:struct}

Recall from \eqref{Vj} that $\subsV^z_{(1)}\subset \C^{\DD\times[N]}$ denotes the span of the $(\rnk+1)(N-1)$ columns of $\LLz_{(1)}$, with the latter viewed as an element of $\matt{\DD\times[N]}{\DD\times[N-1]}{\C}$, and that the columns of $\UU_1$ are a random orthonormal set in $(\subsV^z_{(1)})^\perp$. (In view of \Cref{thm:pseudo} we have, a posteriori, that with high probability $\LLz_{(1)}$ is full rank and the columns of $\UU_1$ in fact comprise a basis for $(\subsV^z_{(1)})^\perp$.)
The aim of this section is to establish the following:

\begin{prop}[Structured bases are rare]
\label{prop:struct}
Assume $\rnk\ge 2$.
Then, we have
\begin{equation}
\P\{ \UU_1 \in \Struct(N^{-\frac12\rnk -10}) \} \ls_{\pol, z} e^{- N}
\end{equation}
where the sets $\Struct(\pstr)$ were defined in \eqref{def:struct}.
\end{prop}

We now conclude the proof of \Cref{thm:pseudo} assuming the above proposition. 
From \Cref{lem:rnk1} we may assume $\rnk\ge 2$. 
From \Cref{lem:norm}, the triangle inequality and sub-multiplicativity of the operator norm we have
$\|\Pol\|_\op = O_\pol(1)$. In particular we may assume $|z|=O_\pol(1)$ since otherwise we obtain the claim by simply lower bounding $\smin(\Pol-z) \ge |z|-\|\Pol\|_\op$. 
Now by \Cref{cor:linearize}, 
\[
\pr\{ \smin(\Pol-z) \le \eps\} \le \pr \{ \smin(\LLz) \le \eps\} 
\]
and from \Cref{lem:bdd} the latter is in turn bounded by
\[
N\pr\big\{\, |\det(\hL_1^z)|\le \eps(|z|+C_\pol)^\rnk\sqrt{N}\,\big\} + e^{-N}. 
\]
Finally, from \Cref{prop:struct}, \eqref{def:W} and \Cref{lem:anti-unstruct}, 
\begin{align*}
\pr\big\{\, |\det(\hL_1^z)|\le \eps(|z|+C_\pol)^\rnk\sqrt{N}\,\big\}
&\ls_\pol e^{- N} +  N^{\frac12( 1+ (\rnk+21)/(\rnk+1))}  (|z|+C_\pol)^{\rnk/(\rnk+1)} \eps^{1/(\rnk+1)} \\
& \ls_\pol e^{- N} + N^{13/3} \eps^{1/(\rnk+1)} ,
\end{align*}
where in the second line we bounded the exponent of $N$ by its maximum value at $\rnk=2$.
The claim follows by combining all of these estimates.
\qed

\subsection{High-level proof of \Cref{prop:struct}}
\label{sec:struct-high}

Our approach is to cut the set $\Struct(\delta)$ (for sufficiently small $\delta = N^{-O_{\pol}(1)}$) into several pieces, and to bound the event that $\UU_1$ lies in each piece by taking union bounds over nets. (Recall the definition of a net from \Cref{sec:facts}.) 
Once we have approximated $\UU_1$ by a some fixed element $\UU$ of a net, our task is then to bound the probability that the columns of $\UU$ are nearly contained in $(\subsV^z_{(1)})^\perp$, i.e.\ that $\tUU^*\col_j(\tLLz)\approx 0$ (the zero matrix in $\Mat_{[0,\rnk]}(\C)$) for each $2\le j\le N$. 
We have
\begin{align}
\tUU^* \col_j(\tLLz) = \sum_{i=1}^N U_{i}^* \Lez_{i,j} &= U_j^*\Kez + \sum_{i=1}^N U_i^*\Le_{i,j}   \notag\\
&=: M_j(\UU,z) + \Walk_j(\UU) \in \Mat_{[0,\rnk]}(\C)	\label{jwalk}
\end{align}
where the matrices $M_j$ are deterministic shifts and 
$\Walk_j(\UU)$ are i.i.d.\ copies of the centered random walk
\begin{equation}
\Walk(\UU) = \sum_{i=1}^N U_i^*\Le_{i}
\end{equation}
with 
\[
\Le_{i} = \Lin(\xi_i^1,\dots, \xi_i^\num) = \frac1{\sqrt{N}} \Lin(\zeta_i^1,\dots,\zeta_i^\num)
\]
as in \eqref{def:Li}, where $\zeta_i^k = \frac1{\sqrt{2}} (g_i^k + \ii h_i^k)$ are i.i.d.\ standard complex Gaussians.
Our task is thus reduced to proving an anti-concentration bound for $\Walk(\UU)$.
Note that whereas in \Cref{sec:unstruct} we were concerned with anti-concentration for the determinant of $W=M_1+\Walk(\UU)$, here we need to show the matrix $\Walk(\UU)$ is anti-concentrated as a random element of the vector space $\Mat_{[0,\rnk]}(\C)$.

Anti-concentration for $\Walk(\UU)$ is most transparent when viewing the matrix as a Gaussian vector in $\C^{(\rnk+1)^2}\cong \C^{[0,\rnk]^2}$. 
Indeed, 
we note from \eqref{def:Li} that $\Le_i$ is a Gaussian linear combination $\frac1{\sqrt{N}}\sum_{\ell=1}^\num \zeta_i^\ell M_\ell$ of matrices $M_\ell$, where $M_\ell^*$ has zeroth row $(\slin_0(\ell),\dots, \slin_\rnk(\ell))$, $\ell$th row $e_\ell$ (the $\ell$th standard basis vector) and all other rows equal to zero. Thus
\[
\Walk(\UU) = \frac1{\sqrt{N}}\sum_{(\ell,i)\in [\num]\times[N]} \zeta_i^\ell M_\ell^*U_i.
\]
Recalling \eqref{def:ws}, the matrix $M_\ell^*U_i$ has zeroth row $(w_i^\ell)^*$, $\ell$th row $(v_i^\ell)^*$, and all other rows equal to zero. 
We can hence express
\begin{equation}	\label{def:walkvec}
\Walk(\UU) = \frac1{\sqrt{N}}\Flat ^* \gvec \in \C^{[0,\rnk]^2} 
\end{equation}
where $\gvec= (\gstep_i^k)_{(k,i)\in [\num]\times[N]}$, and we introduce the \emph{walk matrix}
\begin{equation}	\label{def:flat}
\Flat(\UU) :=
\begin{pmatrix}
Q^1 & U^0 & 0 &\cdots & 0\\
Q^2 & 0 & U^0 & \cdots & 0\\
\vdots& \vdots & &\ddots & \vdots\\
Q^{\rnk}& 0& \cdots & 0 & U^0\\
Q^{\rnk+1}& 0& \cdots & 0 & 0\\
\vdots& \vdots & &\vdots&\vdots \\
Q^\num & 0& \cdots & 0 & 0
\end{pmatrix}
=: (\, \QQ( \UU) \; \RR (U^0) \,) 
\in \matt{[n]\times[N]}{ [0,\rnk]^2}{\C},
\end{equation}
recalling the matrices $Q^\ell$ from \eqref{def:Qell}. 
(See \eqref{outline:WW} in the proof outline for the form of the walk matrix in the case of the anti-commutator polynomial $\pol=x_1x_2+x_2x_1$.)
From \eqref{def:walkvec} we see that the probability the Gaussian vector $\Walk(\UU)$ lands in a Hilbert--Schmidt ball of radius $\eps$ scales like $O(\eps^2)^{\rank(\Flat)}$ for small $\eps$
(we will presently give a quantitative version of this statement).
In particular, from the form of \eqref{def:flat} we see that a lower bound $\smin(U^0)\gs 1$ already guarantees a small ball probability of order $O(\eps^2)^{\rnk(\rnk+1)}$. 
The next lemma shows we may assume such a bound holds.
For a parameter $\pssv >0$, we define the set of ``good'' matrices $\UU$ 
having a well-conditioned zeroth block:
\begin{align}
\good(\pssv ) &:= \{ \,\UU\in \matt{[0,\rnk]\times[N]}{[0,\rnk]}{\C}: \;\smin(U^0) \ge \pssv\,\}. \label{def:good}
\end{align}
Recall that $\LL_{(1)}$ is obtained by removing the first column from each of the $\rnk+1$ blocks of $N$ columns of $\LL$, and that $\UU_1$ depends on $\LL$ only through $\LL_{(1)}$.

\begin{lemma}
\label{lem:block0}
There is constant $c_\pol>0$ depending only on $\pol$ such that for any $\bound\ge 1$ and $z\in \C$, 
\begin{equation}\label{event.block0}
\P\big\{\, \UU_1\notin \good(\pssv )  \,\wedge\, \|\LL_{(1)}\|_\op \le \bound \,\big\} \ls_\pol e^{-N}
\end{equation}	
for any $\pssv <c_\pol/(\bound^2 + |z|)$.
\end{lemma}

We prove this lemma in \Cref{sec:block0}.

By tensorizing an anti-concentration bound for $\Walk(\UU)$ we obtain the following anti-concentration estimate for $\UU^* \LLz_{(1)}$. 
In other words, we bound the probability that the columns of a fixed matrix $\UU$ are almost contained in $\subsV^z_{(1)}$.
The bounds assume $U^0$ is well conditioned, and  depend on how many columns of $\QQ$ are in general position with respect to each other and to the columns of $\RR$.
We use the following notation:
For given $\UU$, denoting the columns of $\QQ$ by $\bq_k$, $k\in [0,\rnk]$, 
we write $\Flat_I$ for the submatrix of $\Flat$ formed by the columns of $\RR$ together with the columns $\{\bq_k\}_{k\in I}$ of $\QQ$ (in the natural order, say, though we note the order of columns will not be important). 
We sometimes write $\Flat_\emptyset :=\RR$.
The column span of $\Flat_I$ is denoted $\subsp_I$, that is:
\begin{align}	
\subsp_I &:= \Span( \subsp_\emptyset, \bq_{k_1},\dots, \bq_{k_r}) \qquad \text{ for } I = \{ k_1,\dots, k_r\}\subset [0,\rnk].	\label{def:subspi}
\end{align} 

\begin{lemma}	\label{lem:mat-ac}
Let  $z\in \C$ and $\pssv>0$, and fix some arbitrary $\UU\in \good(\pssv)$.
Then for any $\eps>0$ and $1\le j\le N$,
\begin{equation}	\label{flat.bound0}
\pr \{ \| \UU^* \LLz_{(1)} \|_\HS \le \eps \} = O_{\pol,\pssv}(\eps^2)^{\rnk(\rnk+1)(N-1)}.
\end{equation}
Furthermore, if additionally it holds that for some $\pflat>0$, $1\le r\le \rnk+1$ and $k_1,\dots, k_r\in [0,\rnk]$ distinct,
\begin{equation}	\label{flat:interest}
\dist(\bq_{k_i}, \subsp_{\{k_1,\dots, k_{i-1}\}} ) \ge \pflat\qquad \forall\, 1\le i\le r,
\end{equation}
(interpreting the left hand side as $\dist(\bq_{k_1}, \subsp_\emptyset)$ when $i=1$), 
then for any $\eps>0$ and $1\le j\le N$,
\begin{equation} \label{mat-ac2}
\pr \{ \|\UU^*\LLz_{(1)} \|_\HS \le \eps \} = O_{\pol,\pssv,\pflat}(\eps^2)^{[\rnk(\rnk+1)+r](N-1)} \,.
\end{equation}
\end{lemma}

We defer the proof to \Cref{sec:mat-ac}.

We get the best anti-concentration estimate in \eqref{mat-ac2} when $r=\rnk+1$, i.e.\ when the full walk matrix $\Flat(\UU)$ is well conditioned. 
The following lemma shows we may assume this is the case.
For $\pflat>0$ we define the set
\begin{equation}
\badR(\pflat) := \bigcup_{k\in [0,\rnk]} \big\{\, \UU \in \matt{[0,\rnk]\times[N]}{[0,\rnk]}{\C} : \dist(\bq_k, \subsp_{[0,\rnk]\setminus\{k\}}) <\pflat \,\big\}
\end{equation}

\begin{lemma}[Walks usually have full rank]
\label{lem:fullrank}
For any $\pssv>0$ and $\bound\ge1$ there exists $\pflat_\star(\pol,\pssv,\bound)>0$ such that for any $\pflat\in (0,\pflat_\star]$ and $z\in \C$,
\[
\pr\{\, \UU_1\in \badR(\pflat) \cap \good(\pssv) \; \wedge\; \|\LLz\|_\op \le \bound \, \} 
\ls_{\pol} e^{-N}.
\]
\end{lemma}

We prove this lemma in \Cref{sec:struct.rank}. While the lemma allows us to apply \eqref{mat-ac2} with $r=\rnk+1$, we will need to apply this estimate with smaller values of $r$ in the proof.

In view of Lemmas \ref{lem:block0} and \ref{lem:fullrank} it only remains to bound the probability that $\UU_1$ lies in the set
\begin{equation}	
\cE = \cE(\pstr,\pssv,\pflat) := \Struct(\pstr) \cap \good(\pssv)\cap \badR(\pflat)^c.
\end{equation}
We will bound the probability that $\UU_1$ lies in $\cE$ by combining the the anti-concentration estimate \eqref{mat-ac2} (taking $r=\rnk+1$) with a union bound over a suitable net, provided by the following:

\begin{lemma}\label{lem:struct.net}
For any $\pssv>0$, $\pnet\in (0,1)$ and $0<\delta\le \pnet^2 N^{-\frac{\rnk}2-2}$, $\cE$ has a (Hilbert--Schmidt) $\pnet$-net $\cN\subset \cE$ of size 
\begin{equation}	\label{cNbound}
|\cN| = O_{\pol,\pssv}(1)^N N^{(2\rnk+3)N} O(1/\pnet^2)^{(\rnk^2+\frac32\rnk +\frac32)N}.
\end{equation}
\end{lemma}

\begin{remark}
The proof in fact yields a net for any subset of $\Struct(\delta)\cap \good(\pssv)$, which is why the bound \eqref{cNbound} is independent of 
$\pflat$.
\end{remark}

The proof of \Cref{lem:struct.net} is deferred to \Cref{sec:struct.net}.
We now conclude the proof of \Cref{prop:struct} assuming the above lemmas.
Let $\pstr>0$ to be chosen sufficiently small in the course of the proof.
From \Cref{lem:norm} and the triangle inequality we have that $\|\LL\|_\op= O_\pol(1)$ with probability $1-O_\pol(e^{-N})$, and hence $\|\LL_{(1)}\|_\op, \|\LLz\|_\op,\|\LLz_{(1)}\|_\op=O_{\pol,z}(1)$ with probability $1-O_\pol(e^{-N})$.
Applying the union bound, for $\alpha,\beta>0$ and some $\bound=O_{\pol,z}(1)$ we have
\begin{align*}
\pr \{ \UU_1\in \Struct(\pstr) \} 
&\le \pr\{ \UU_1 \in \cE(\pstr,\pssv,\pflat) \;\wedge\; \| \LLz_{(1)}\|_\op \le \bound\}\\
&\quad+ \pr\{ \UU_1\notin\good(\pssv) \;\wedge\; \|\LL_{(1)}\|_\op\le \bound\}\\
&\quad+ \pr\{ \UU_1\in \good(\pssv)\cap \badR(\pflat) \; \wedge\; \|\LLz\|_\op \le \bound\} 
+ O_\pol(e^{-N}).
\end{align*}
For this choice of $\bound$ we now choose $\pssv\gs_{\pol,z}1$ satisfying the constraint in \Cref{lem:block0}, followed by $\pflat\gs_{\pol,z}1$ satisfying the constraint in \Cref{lem:fullrank}. From those lemmas we then have
\begin{equation}
\pr \{ \UU_1\in \Struct(\pstr) \} \le \pr\{ \UU_1\in \cE(\pstr,\pssv,\pflat) \; \wedge\; \|\LLz\|_\op \le \bound\}  + O_{\pol}(e^{-N}).
\end{equation}
Let $\pnet>0$ to be chosen later and let $\cN$ be as in \Cref{lem:struct.net} (assuming $\pstr$ is sufficiently small depending on the choice of $\pnet$). 
By approximating any realization of $\UU_1\in \cE$ by some $\UU\in \cN$ with $\|\UU_1-\UU\|_\HS\le \pnet$ and using that 
\[
\|(\UU_1-\UU)^*\LLz_{(1)}\|_\HS \le \|\LLz\|_\op \|\UU_1-\UU\|_\HS 
\]
we have, since $ \UU_{1}^*\LLz_{(1)}$ vanishes everywhere by definition, 
\begin{align}
 \pr\{ \UU_1\in\cE \; \wedge\; \|\LLz\|_\op \le \bound\}
 &\le \pr\{ \exists \UU\in \cE: \UU^*\LLz_{(1)} = 0 \;\wedge\;  \|\LLz\|_\op \le \bound \}  \notag \\
 &\le \pr\{ \exists \UU\in \cN:  \|\UU^*\LLz_{(1)}\|_\HS\le \bound\pnet \}.	\label{prop7.net1}
\end{align}
Now applying the union bound and \eqref{mat-ac2} with $r=\rnk+1$, using that $\cN\subset \cE\subset\badR(\beta)^c$, we have
\begin{align}
 \pr\{ \UU_1\in\cE \; \wedge\; \|\LLz\|_\op \le \bound\}
 &\le |\cN| O_{\pol,z}(\pnet^2)^{(\rnk+1)^2(N-1)}	\label{prop7.net2} \\
 &\le O_{\pol,z}(1)^N N^{(2\rnk+3)N} \pnet^{2[(\rnk+1)^2 - \rnk^2 -\frac32\rnk-\frac32 ] N - O_\pol(1)}	\notag\\
 &=O_{\pol,z}(1)^N N^{(2\rnk+3)N} \pnet^{ (\rnk-1) N - O_\pol(1)}.	\notag
\end{align}
Recalling our assumption that $\rnk\ge2$, one verifies the last expression is at most 
\[
O_{\pol,z}(1)^N (c')^{N-O_\pol(1)} \ls_\pol e^{-N}
\]
if we take
\[
\pnet = c' N^{-(2\rnk+3)/(\rnk-1)}
\]
for a sufficiently small constant $c'(\pol,z)>0$. Now one verifies that the condition of \Cref{lem:struct.net} is satisfied for all $N\ge N_0$ for a sufficiently large constant $N_0(\pol,z)>0$ when $\pstr = N^{-\frac{\rnk}2-10}$. 
This completes the proof of \Cref{prop:struct} assuming Lemmas \ref{lem:block0}, \ref{lem:mat-ac}, \ref{lem:fullrank} and \ref{lem:struct.net}.
\qed

\subsection{Proof of \Cref{lem:block0}: Reduction to bases with a well-conditioned zeroth block}
\label{sec:block0}

We express $\LLz_{(1)}$ in block form as follows:
\[
\LLz_{(1)} = 
    \left(
      \begin{array}{c|c|c|c|c}
        \begin{array}{c}  \row_1(\Ylin_0') \\ \Ylin_0'' + (\shift-z) \id_{N-1} \end{array}  & 
        \begin{array}{c}  \row_1(X_1') \\ X_1'' \end{array} & 
        \begin{array}{c}  \row_1(X_2') \\ X_2'' \end{array} &
        \quad\cdots\quad &
        \begin{array}{c}  \row_1(X_\rnk') \\ X_\rnk'' \end{array} \\
        \hline
        \begin{array}{c}  \row_1(\Ylin_1') \\ \Ylin_1'' \end{array} & 
        \begin{array}{c}  0_{1,N-1} \\ -\id_{N-1} \end{array}  & 
        0 &
        \quad\cdots\quad &
        0 \\
        \hline
        \begin{array}{c}  \row_1(\Ylin_2') \\ \Ylin_2'' \end{array} & 
        0 & 
        \begin{array}{c}  0_{1,N-1} \\ -\id_{N-1} \end{array} 
         & &  \vdots \\
        \hline  
        & & & & \\
        \vdots & \vdots & & \ddots & 0 \\
        & & & & \\
        \hline
	\begin{array}{c}  \row_1(\Ylin_\rnk') \\ \Ylin_\rnk'' \end{array} & 
        0 & \quad\cdots\quad & 0 &
        \begin{array}{c}  0_{1,N-1} \\ -\id_{N-1} \end{array}  
      \end{array}
    \right).
\]
Here, $\Ylin_k'$ (resp.\ $X_k'$) is the matrix obtained by removing the first column from $\Ylin_k$ (resp.\ $X_k$), while $\Ylin_k''$ (resp.\ $X_k''$) is the $N-1\times N-1$ matrix obtained by removing the first row and column from $\Ylin_k$ (resp.\ $X_k$). $0_{1,N-1}$ is the row vector of $N-1$ zeros. 

Let
\begin{equation}
\wh{\Ylin} 
= 
\begin{pmatrix} \row_1(\Ylin_1') \\ \vdots \\ \row_1(\Ylin_\rnk') \end{pmatrix} \in \Mat_{\rnk, N-1}(\C).
\end{equation}
We first show that
\begin{equation}	\label{wtS.bound}
\pr\{ \smin(\wh{\Ylin}) < c_\pol'\} \ls_\pol e^{- N}
\end{equation}
if $c_\pol'>0$ is sufficiently small. 
On the event that $\smin(\wh{\Ylin})<c_\pol'$, there exists $x\in \sph^{\rnk-1}$ such that 
\[
c_\pol' > \| x^* \wh{\Ylin} \|_2  = \Big\| \sum_{k\in [\rnk]} x(k) \sum_{\ell\in [\num]} \overline{\slin_{k,\ell}} \row_1(X_\ell') \Big\|_2 = \|y(1)\row_1(X'_1) + \cdots + y(\num) \row_1(X'_\num) \|_2
\]
where $y:= \sum_{k=1}^\rnk x(k) \overline{\slin_{k}}$. Since $\slin_1,\dots, \slin_\rnk$ are mutually orthogonal and nonzero it follows that $\|y\|_2 \gs_\pol1$, and in particular $y$ must have a coordinate $y(\ell_0)$ of size $|y(\ell_0)|\gs_\pol 1$.
By projecting $\row_1(X_{\ell_0}')$ to the orthocomplement of $\wt{\subsW}_{(\ell_0)}=\Span\{ \row_1(X'_\ell)\}_{\ell\ne\ell_0}$ we see that $\dist( \row_1(X'_{\ell_0}), \wt{\subsW}_{(\ell_0)}) \ls_\pol c_\pol'$. 
Taking a union bound to fix $\ell_0$ and conditioning on the rows $\row_1(X'_\ell)$ with $\ell\ne\ell_0$, we are left with bounding the probability that the vector $\row_1(X_{\ell_0}')$ is within distance $O_\pol(c_\pol')$ of a fixed subspace of $\C^{N-1}$ of dimension $\num-1$. Since $X_\ell$ has i.i.d.\ complex Gaussian entries of variance $1/N$, by rotational invariance of the Gaussian measure we may assume $\wt{\subsW}_{(\ell_0)}$ is the coordinate subspace spanned by $e_2,\dots, e_{\num}$, and we have reduced to bounding the probability of the event that
\[
\sum_{j=\num+1}^N |X_{\ell_0}(1,j)|^2  = O_\pol(c_\pol'). 
\]
Now we can take $c_\pol'$ sufficiently small to make the right hand side smaller than any fixed constant, so that this event has probability at most $e^{-N}$ as soon as $N\ge 10\num$, say, and \eqref{wtS.bound} follows.

We now fix for the remainder of the proof an arbitrary realization of $\LL_{(1)}$ such that $\|\LL_{(1)}\|_\op \le B$ and $\smin(\wh{\Ylin}) \ge c_\pol'$.
This fixes the subspace $(\subsV^z_{(1)})^\perp$. 
We claim that
\begin{equation}	\label{block0.empty}
\big\{\, \UU=(\bu_0,\dots, \bu_\rnk) : \bu_0, \bu_1,\dots, \bu_\rnk \text{ orthonormal in } (\subsV^z_{(1)})^\perp \,\big\}
\cap \good(\pssv )^c = \emptyset
\end{equation}
for $\pssv <c_\pol/(B^2+|z|)$. 
From this the lemma clearly follows.

Turning to prove \eqref{block0.empty}, towards a contradiction we let $\UU$ be an arbitrary element of 
the left hand side of \eqref{block0.empty}.
Then there exists a unit vector $\bu\in (\subsV^z_{(1)})^\perp$ such that $\|u^0\|_2<\pssv $ (recall that $u^0,\dots, u^\rnk$ denote the blocks of $N$ coordinates of $\bu$). For each $k$ we write $u^k = (u^k(1), \check{u}^k)$, i.e.\ $\check{u}^k = (u^k(2),\dots, u^k(N))$.
Fixing such a $\bu$, we note that $\bu\in (\subsV^z_{(1)})^\perp$ is equivalent to the statement that $\bu^*$ is a left null vector for $\LLz_{(1)}$. 
In particular, for each $1\le k\le \rnk$ we have
\begin{equation}
\overline{u^0(1)} \row_1(X_k') + (\check{u}^0)^* X_k'' = (\check{u}^k)^*.
\end{equation}
Since $\|\LL_{(1)}\|_\op \le B$ and $\|u^0\|_2<\pssv $ this implies
\begin{equation}	\label{ucheck.bd}
\|\check{u}^k\|_2 \le 2B/\pssv , \quad \forall \,1\le k\le \rnk.
\end{equation}

Again from the fact that $\bu^*$ is a left null vector for $\LLz_{(1)}$, we have
\[
\sum_{k=1}^\rnk \overline{u^k(1)} \row_1(\Ylin_k') 
= - \overline{u^0(1)} \row_1(\Ylin_0') - (\check{u}^0)^* (\Ylin_0''+ (\shift-z)\id_{N-1}) - \sum_{k=1}^\rnk (\check{u}^k)^* \Ylin_k''
\]
for each $1\le k\le \rnk$.
Writing $\check{u}_1 = (u^1(1), \dots, u^\rnk(1))$, the left hand side has norm at least $\smin(\wh{\Ylin})\|\check{u}_1\|_2 \ge c_\pol'\|\check{u}_1\|_2$. Since $\LL_{(1)}$, and hence any of its submatrices, has norm at most $B$, together with \eqref{ucheck.bd} and $\|u^0\|_2<\pssv $ this implies by the triangle inequality that the right hand side is $O_\pol(\pssv (B^2+|z|))$. 
Putting these bounds together we have
\[
\|\check{u}_1\|_2 \ls_\pol (B^2 + |z|)\pssv .
\]
Combined with \eqref{ucheck.bd} and $\|u^0\|_2<\pssv $ we conclude
\[
1 = \|\bu\|_2^2 = \|u^0\|_2^2 + \|\check{u}_1\|_2^2 + \sum_{k=1}^\rnk \|\check{u}^k\|_2^2 
\le \pssv^2 ( 1+ O_\pol(B^2 + |z|)^2)
\]
and we obtain a contradiction for $\pssv<c_\pol /(B^2+|z|)$ if $c_\pol$ is taken sufficiently small. 
This implies 
the the left hand side of \eqref{block0.empty}
is empty for such a choice of $\pssv$, and completes the proof.

\subsection{Proof of \Cref{lem:mat-ac}: Anti-concentration for matrix random walks}
\label{sec:mat-ac}

\begin{lemma}
\label{lem:anti-vec}
Let $M\in \Mat_{d,d'}(\C)$ and let $Z= (\zeta_1,\dots, \zeta_{d'})$ be a vector of i.i.d.\ standard complex Gaussians. For $I\subseteq[d]$ let $M_I\in \Mat_{|I|,d'}(\C)$ be the submatrix of $M$ formed by the rows indexed by $I$.
If $M_I$ has full row rank for some $I\subseteq[d]$ of size $d_0$, then
\begin{equation}	\label{antic}
\sup_{w\in \mathbb C^d}\mathbb P(\|MZ-w\|_2\le \varepsilon)\le \frac{O(\eps^2)^{d_0}}{\det(M_IM_I^*)} \,.
\end{equation}
\end{lemma}

\begin{proof}
By projecting $MZ$ to the coordinate subspace $\C^I$ we may assume $I=[d]$.
Let $M$ have singular value decomposition $U\Sigma V^*$ with $U$ and $V$ square unitary matrices of respective dimensions $d$ and $d'$, and with $\Sigma= \diag(\sigma_1(M), \dots, \sigma_d(M))$. We have
\begin{align*}
\sup_{w\in \mathbb C^d}\mathbb P(\|MZ-w\|_2\le \varepsilon) 
&= \sup_{w\in \mathbb C^d}\mathbb P(\|\Sigma V^*Z-U^*w\|_2\le \varepsilon)\\
&= \sup_{w\in \mathbb C^d}\mathbb P(\|\Sigma V^*Z-w\|_2\le \varepsilon) \\
&= \sup_{w\in \mathbb C^d}\mathbb P(\|\Sigma Z-w\|_2\le \varepsilon)
\end{align*}
where we used the invariance of the Euclidean norm and the distribution of $Z$ under unitary transformations of $\C^d$ and $\C^{d'}$, respectively. 
Now we have
\[
\P(\|\Sigma Z-w\|_2\le \varepsilon) = \pr \bigg( \sum_{j=1}^d |\sigma_j(M)\zeta_j-w_j|^2 \le \eps^2\bigg)
\le \prod_{j=1}^d \pr(  |\sigma_j(M)\zeta_j-w_j| \le \eps)
\]
and the latter quantity is $O(\eps)^{2d}/\prod_{j=1}^d \sigma_j(M)^2$ by the boundedness of the standard Gaussian density. 
The claim follows.
\end{proof}

\begin{lemma}[Tensorization of anti-concentration (cf.\ {\cite{RuVe:LO}})]
\label{lem:tensorize}
Suppose that $\xi_1,\dots, \xi_m$ are independent non-negative random variables such that, for some $\eps_1,\eps_2>0$, $\pr\{ \xi_i\le\eps_1\} \le \eps_2$ for all $1\le i\le m$.
Then
\[
\pr\Big\{\, \sum_{i=1}^m \xi_i^2 \le \eps_1^2 m \,\Big\} = O(\eps_2)^m.
\]
\end{lemma}

We now prove \Cref{lem:mat-ac}.
Since $\UU\in \good(\pssv)$ we have $\smin(\RR)= \smin(U^0)\ge \pssv$, and hence
\begin{equation}
\det(\RR^*\RR) \gs_{\pol,\pssv}1.
\end{equation}
Letting $\eps>0$ be arbitrary,
from \eqref{def:walkvec} and \Cref{lem:anti-vec} with $M= \Flat^*$, $M_I=\RR^*$ we obtain 
\[
\sup_{M\in\Mat_{\DD}(\C)} \pr \big\{\, \| M + \Walk(\UU) \|_\HS \le \eps N^{-1/2} \,\big\} =O_{\pol,\pssv}(\eps^2)^{\rnk(\rnk+1)}.
\]
In particular, from \eqref{jwalk} we have that for each $2\le j\le N$,
\[
\pr \big\{\, \|\tUU^*\col_j(\tLLz) \|_\HS \le \eps N^{-1/2} \,\big\}  = O_{\pol,\pssv}(\eps^2)^{\rnk(\rnk+1)}.
\]
Since the columns $\col_j(\tLLz)$ are independent, from \Cref{lem:tensorize} we obtain
\[
\pr \{ \, \|\UU^*\LLz_{(1)} \|_\HS \le \eps \,\}
= \pr \bigg\{ \, \sum_{j=1}^N \|\tUU^*\col_j(\tLLz)\|_\HS^2 \le \eps^2 \,\bigg\}
=O_{\pol,\pssv}(\eps^2)^{\rnk(\rnk+1)(N-1)}
\]
which is \eqref{flat.bound0}.

For \eqref{mat-ac2}, under the assumptions $\smin(U^0)\ge \pssv$ and \eqref{flat:interest}, with $I= \{ k_1,\dots, k_r\}$ we have
\[
\det(\Flat_I^*\Flat_I) \ge \det(\RR^*\RR) \cdot \pflat^{2r} \gs_{\pol,\pssv,\pflat}1,
\]
(recalling the notation around \eqref{def:subspi} and using the fact that 
\[ \det(\Flat_{I_{i+1}}^*\Flat_{I_{i+1}}) = |\dist(\bq_{k_{i+1}}, \subsp_{ I_i})|^2 \det(\Flat_{I_{i}}^*\Flat_{I_{i}}) \]
 for $I_i = \{k_1,...,k_{i} \}$) and the proof concludes by following the same lines as we did for \eqref{flat.bound0}.
\qed

\subsection{Proof of \Cref{lem:fullrank}: Reduction to matrix walks of full rank}
\label{sec:struct.rank}

We will argue iteratively, incrementing the rank parameter $r$ in \eqref{flat:interest} from 0 to $\rnk$.
For brevity, in this section we denote the dilated Hilbert--Schmidt ball
\[
\BALL := 2\sqrt{\rnk+1}\cdot \matball{[0,\rnk]\times[N]}{[0,\rnk]}
\]
(recall our notation $\matball{S}{T}\subset\matt{S}{T}{\C}$ for the closed unit Hilbert--Schmidt ball). Note that $\UU_1\in \BALL$ almost surely since its $\rnk+1$ columns are unit vectors.
For $\pflat_1>0$ let
\begin{equation}
\badR_\emptyset(\pflat_1) := \big\{\, \UU \in \BALL : \dist(\bq_k, \subsp_\emptyset) <\pflat_1 \; \forall k\in [0,\rnk] \,\big\}
\end{equation}
and for $\pflat_0\ge\pflat_1>0$, $1\le r\le \rnk$ 
and
$k_1,\dots, k_r\in [0,\rnk]$ distinct, 
\begin{align}
\badR_{(k_1,\dots, k_r)} (\pflat_0) &:= \big\{\, \UU \in \BALL  : \dist(\bq_{k_i}, \subsp_{\{k_1,\dots, k_{i-1}\}}) \ge \pflat_0 \quad \forall \;1\le i\le r \, \big\}\, \,		\label{def:badRk}\\
\tbad_{(k_1,\dots, k_r)} (\pflat_0,\pflat_1) &:= \badR_{(k_1,\dots, k_r)}(\pflat_0) \setminus \bigcup_{k_{r+1}\notin \{ k_1,\dots,k_r\}} \badR_{(k_1,\dots, k_{r+1})}(\pflat_1).	\label{def:tbad}
\end{align}
(interpreting $\subsp_{\{k_1,\dots, k_{i-1}\}}$ as $\subsp_\emptyset$ for $i=1$). 

\begin{claim}	\label{claim:net1}
For any $\pssv\in (0,\frac12]$ there exists $c_0(\pol,\pssv)>0$ such that for any $\pflat_1\in (0,c_0]$, there is a set 
$\cN_\emptyset(\pflat_1)\subset \BALL\cap\good(\pssv)$ with
\begin{equation}
|\cN_\emptyset(\pflat_1)| \le O_\pol(1/\pflat_1^2)^{(\rnk+1)N }
\end{equation}
and such that for any $\UU\in \badR_\emptyset(\pflat_1)\cap \good(\pssv)$ there exists $\hUU\in \cN_\emptyset(\pflat_1)$ with
\begin{equation}
\| \UU - \hUU\|_\HS \ls_{\pol,\pssv} \pflat_1.
\end{equation}
\end{claim}

\begin{claim}	\label{claim:net2}
For any $\pssv,\pflat_0\in(0,\frac12]$ there exists $c_1(\pol,\pssv,\pflat_0)>0$ such that for any $\pflat_1\in (0,c_1]$, $1\le r\le \rnk$ and distinct $k_1,\dots, k_r\in \DD$, there is a set 
$\cN_{(k_1,\dots, k_r)}(\pflat_0,\pflat_1)\subset \badR_{(k_1,\dots,k_r)}(\pflat_0/2)\cap\good(\pssv)$ with
\begin{equation}	\label{netr.size}
|\cN_{(k_1,\dots, k_r)}(\pflat_0,\pflat_1)| \le O_\pol(1/\pflat_1^2)^{((r+1)\rnk+1)N }
\end{equation}
and such that for any $\UU\in \tbad_{(k_1,\dots,k_r)}(\pflat_0,\pflat_1)\cap\good(\pssv)$ there exists $\hUU\in \cN_{(k_1,\dots, k_r)}(\pflat_0,\pflat_1)$ with
\begin{equation}	\label{UUclose}
\| \UU - \hUU\|_\HS \ls_{\pol,\pssv,\pflat_0} \pflat_1.
\end{equation}
\end{claim}

We now conclude the proof of \Cref{lem:fullrank} on the above claims.
By the same lines we used in the proof of \Cref{prop:struct} (cf.\ \eqref{prop7.net1}--\eqref{prop7.net2}) we have from \Cref{claim:net1} and \eqref{flat.bound0} that for $\beta_0\le c_0(\pol,\pssv)$, 
\begin{align*}
\pr\{\UU_1 \in \badR_\emptyset( \beta_0)\cap \good(\pssv) \;\wedge\; \|\LLz\|_\op\le \bound\}
&\le |\cN_\emptyset(\beta_0)| O_{\pol,\pssv}(\bound^2\beta_0^2)^{\rnk(\rnk+1)(N-1)}\\
&\le O_{\pol,\pssv,B}(1)^N O(\beta_0^2)^{[\rnk(\rnk+1) - (\rnk+1)]N-O_\pol(1)}.
\end{align*}
Since $\rnk(\rnk+1) - (\rnk+1) = \rnk^2-1 \ge 1$ (recall our assumption that $\rnk\ge2$), it follows that there exists $\beta_0=\beta_0(\pol,\pssv,\bound)>0$ such that
\begin{equation}	\label{Bzero}
\pr\{\UU_1 \in \badR_\emptyset( \beta_0)\cap \good(\pssv) \;\wedge\; \|\LLz\|_\op\le \bound\}
\ls_{\pol} e^{-N}.
\end{equation}
Similarly, from \Cref{claim:net2} and \eqref{flat:interest}, for any $\gamma_0\in (0,\frac12]$ and $\gamma_1\le c_1(\pol,\pssv,\gamma_0)$, any $1\le r\le \rnk$ and distinct $k_1,\dots,k_r\in [0,\rnk]$,
\begin{align*}
&\pr\{\UU_1 \in \tbad_{(k_1,\dots, k_r)}( \gamma_0,\gamma_1)\cap \good(\pssv) \;\wedge\; \|\LLz\|_\op\le \bound\}\\
&\qquad\qquad\le |\cN_{(k_1,\dots, k_r)}(\gamma_0,\gamma_1)| O_{\pol,\pssv,\gamma_0}(\bound^2\gamma_1^2)^{ [\rnk(\rnk+1)+r](N -1)}\\
&\qquad\qquad\le O_{\pol,\pssv,\bound,\gamma_0} (1)^N O(\gamma_1^2)^{[\rnk(\rnk+1) + r - (r+1)\rnk-1] N - O_\pol(1)}.
\end{align*}
Since $\rnk(\rnk+1) + r - (r+1)\rnk-1 = \rnk(\rnk-r) + r-1 \ge 1$ for all $\rnk,r$ with $1\le r\le \rnk$ and $\rnk\ge 2$, it follows that there exists 
$\gamma_1(\pol,\pssv,\bound,\gamma_0)\in (0,\gamma_0]$ such that
\begin{equation}	\label{Bstep}
\pr\{\UU_1 \in \tbad_{(k_1,\dots, k_r)}( \gamma_0,\gamma_1)\cap \good(\pssv) \;\wedge\; \|\LLz\|_\op\le \bound\} \ls_{\pol} e^{-N}.
\end{equation}
Now we recursively obtain a sequence $\beta_0\ge \beta_1\ge \cdots\ge \beta_\rnk=:\pflat_\star>0$ with $\beta_0=\beta_0(\pol,\pssv,\bound)$ as in \eqref{Bzero}, and $\beta_k=\gamma_1(\pol,\pssv,\bound,\beta_{k-1})$ for $1\le k\le \rnk$. 
We cover
\begin{align*}
&\badR(\pflat_\star)\cap \BALL \subseteq\\
&\qquad\badR_\emptyset(\beta_0)\, \cup \bigcup_{\substack{k_1,\dots, k_\rnk\in [0,\rnk] \\\text{ distinct}}}
\tbad_{(k_1)}(\beta_0,\beta_1) \,\cup\, \tbad_{(k_1,k_2)}(\beta_1,\beta_2)\, \cup\,\cdots \,\cup\,\tbad_{(k_1,\dots, k_\rnk)}(\beta_{\rnk-1}, \beta_\rnk).
\end{align*}
Applying the union bound followed by \eqref{Bzero} and \eqref{Bstep} with $(\gamma_0,\gamma_1) = (\beta_{k-1},\beta_k)$ for $1\le k\le \rnk$, we obtain 
\begin{align*}
&\pr\{\, \UU_1\in \badR(\pflat)\cap \good(\pssv) \; \wedge\; \|\LLz\|_\op\le \bound \,\} \ls_\pol e^{-N}
\end{align*}
for any $0<\pflat\le \pflat_\star$. 
\qed\\

It remains to prove Claims \ref{claim:net1} and \ref{claim:net2}.
The proof of \Cref{claim:net1} is essentially contained in that of \Cref{claim:net2} -- we give the proof of the latter and describe at the end which steps can be skipped to establish the former.

We will use the following elementary lemma.

\begin{lemma}[Stability of distances to subspaces]
\label{lem:stab}
Let $1\le k\le m$.
Let $b,a_1,\dots, a_k,$ $b', a'_1,\dots, a'_k\in \C^m$ and $A=(a_1,\dots, a_k)$.
Suppose $\smin(A) \ge \alpha>0$, $\|b-b'\|\le \eps$, and $\|a_i-a'_i\|_2\le \eps$ for each $i\in [k]$. If $\eps\le \frac12k^{-1/2}\alpha$, then
\[
\big| \dist(b, \Span(a_1,\dots, a_k) ) - \dist(b', \Span(a'_1,\dots, a'_k) \big| 
\ls  (1+\|b\|_2)k\eps/\alpha.
\]
\end{lemma}

\begin{proof}
From the triangle inequality it suffices to show
\begin{equation} \label{dist.bd}
\big| \dist(b, \Span(a_1,\dots, a_k) ) - \dist(b, \Span(a'_1,\dots, a'_k) \big| 
\ls k \eps\|b\|_2/\alpha.
\end{equation}
Set $A^{(0)} = A$ and for each $1\le j\le k$ let $A^{(j)}$ be obtained by replacing the first $j$ columns of $A$ with $a_1',\dots, a_j'$, so that $A^{(k)} = A'$. 
By expanding the left hand side of \eqref{dist.bd} as a telescoping sum and applying the triangle inequality, we see it suffices to show 
\begin{equation}
|\dist(b, \Span(A^{(j)})) - \dist(b, \Span(A^{(j-1)}))| \le \frac{4 \|b\|_2\eps}{\smin(A)}
\end{equation}
for each $1\le j\le k$. 

Fix such a $j$, and let $V$ be the span of all columns of $A^{(j)}$ but the $j$th one. Let $\tilde b, \tilde a_j, \tilde a_j'$ denote the projections of $b, a_j, a_j'$ to $V^\perp$. Then the left hand side above is
\begin{align*}
|\dist(\tilde b, \langle \tilde a_{k}\rangle)-\dist(\tilde b, \langle \tilde a_{k}'\rangle)|&
= \bigg| \frac{ \|\tilde b\wedge\tilde a_j \|_2}{ \|\tilde a_j\|_2} -  \frac{ \|\tilde b\wedge\tilde a_j' \|_2}{ \|\tilde a_j'\|_2} \bigg| \\
&\le \frac{1}{\|\tilde a_j\|_2} \big| \|\tilde b \wedge \tilde a_j\|_2 -\|\tilde b\wedge\tilde a_j'\|_2 \big|
+  \bigg| \frac{\|\tilde a_j'\|_2}{\|\tilde a_j\|_2} -1 \bigg|  \frac{ \|\tilde b\wedge\tilde a_j' \|_2}{ \|\tilde a_j'\|_2}  \\
&\le \frac{1}{\|\tilde a_j\|_2}\|\tilde b\|_2 \eps + \frac{\eps}{\|\tilde a_j\|_2} \|\tilde b\|_2 \le \frac{2\eps \|b\|_2}{\|\tilde a_j\|_2}.
\end{align*}
Noting that $\|\tilde a_j\|_2 = \dist(a_j, V)$, we see it suffices to show
\begin{equation}	\label{dist.goal1}
\dist(a_j, V) \ge \smin(A)/2.
\end{equation}
Now for some $c_1,\dots, c_{j-1}, c_{j+1}, \dots, c_k\in \C$ we have
\[
\dist(a_j, V) = \Big\| a_j - \sum_{i\ne j} c_i \col_i(A^{(j)}) \Big\|_2 \ge \smin(A^{(j)}) \Big( 1+ \sum_{i\ne j} |c_i|^2 \Big)^{1/2} \ge \smin(A^{(j)}).
\]
Now since $\|A^{(j)}-A\|_\op \le \|A^{(j)}-A\|_\HS \le \eps\sqrt{k}$ for all $1\le j\le k$, 
\[
\smin(A^{(j)}) \ge \smin(A) - \|A^{(j)}-A\|_\op \ge \smin(A) - \eps\sqrt{k} 
\]
and \eqref{dist.goal1} follows from the previous two displays and our assumption on $\eps$. 
\end{proof}

\begin{proof}[Proof of \Cref{claim:net2}]
The key property of $\Flat=\Flat(\UU)$ is that the columns of $\QQ$ are determined by those of $\UU$. Indeed, let $\Slin\in \matt{[\num]}{[0,\rnk]}{\C}$ have columns $\slin_0,\dots,\slin_\rnk$ (cf.\ \eqref{def:Li}) and let $\tSlin\in \matt{[0,\num]}{ [0,\rnk]}{\C}$ be the result of adding the standard basis vector $e_1=(1,0,\dots, 0)\in \C^{\num}$ as zeroth row. 
We have
\begin{equation}	\label{QU}
\begin{pmatrix} U^0\\ \QQ \end{pmatrix} = (\tSlin\otimes \id_N) \UU\,, \qquad \text{ in particular }\; \QQ = (\Slin\otimes \id_N) \UU.
\end{equation}
Since $\slin_1,\dots, \slin_\rnk$ are linearly independent, it follows that $\tSlin$ has full column rank, so 
\begin{equation}	\label{def:tS+}
\UU = (\tSlin^+\otimes \id_N) \begin{pmatrix} U^0\\ \QQ \end{pmatrix} 
\end{equation}
where $\tSlin^+ = (\tSlin^*\tSlin)^{-1}\tSlin^*$ is the Moore--Penrose pseudoinverse of $\tSlin$. 

For $U^0= (u_0^0, \dots, u_\rnk^0)\in \matt{[N]}{ [0,\rnk]}{\C}$ 
and $\bu'_1,\dots, \bu'_r\in \C^{[\rnk]\times[N]}$ with $r\le \rnk$, we write 
$\subsp^{(r)}(U^0, \bu'_1,\dots, \bu'_r)$ for the span of the columns of $\RR(U^0)$ together with $(S\otimes \id_N)\bu_1,\dots, (S\otimes \id_N)\bu_r$, where 
\[
\bu_k = \begin{pmatrix} u^0_k\\ \bu'_k \end{pmatrix}.
\]
In particular, for $I=\{k_1,\dots, k_r\}\subset[0,\rnk]$ we have
\begin{equation}	\label{Wk1}
\subsp_I(\UU) = \subsp^{(r)}(U^0, \bu'_{k_1},\dots,\bu'_{k_r}).
\end{equation} 
We introduce this notation to make it clear that this subspace is fixed by the partial data $U^0, \bu'_{k_1},\dots,\bu'_{k_r}$, which will be crucial for the net construction.

For any $\pflat\in (0,1/2]$ we let $\Sigma_0(\pflat,\pssv)$ be a $\pflat$-net 
for the set 
\[
\big\{\, U\in 2\sqrt{\rnk+1}\cdot\matball{[N]}{[0,\rnk]} : \smin(U) \ge \alpha\,\big\}
\]
and let $\Sigma_1(\pflat)$ be a $\pflat$-net for $\ball^{[\rnk]\times[N]}$ (the unit ball in $\C^{[\rnk]\times[N]}$). 
(Both nets are taken with respect to the Euclidean metric.)
For each $\hU^0\in \Sigma_0(\pflat, \pssv)$, $\hbu'_1,\dots, \hbu'_r\in \Sigma_1(\pflat)$ and any $R\ge1$, we let $\Sigma^{(r)}_{\hU^0,\hbu'_1,\dots,\hbu'_r}(\pflat,R)$ be a (Euclidean) $\pflat$-net for the ball of radius $R$ in $\subsp^{(r)}(\hU^0, \hbu'_1,\dots,\hbu'_r)$.
For $r=0$ we let $\Sigma^{(0)}_{\hU^0}(\pflat,R)$ be a $\pflat$-net for the ball of radius $R$ in $\subsp^{(0)}(\hU^0)$.
By \Cref{lem:net} we may choose these nets so that
\begin{align}	
|\Sigma_0(\pflat,\pssv)| &\le O_\pol(1/\pflat^2)^{(\rnk+1) N}, \label{netbounds0}\\
|\Sigma_1(\pflat)| &\le O(1/\pflat^2)^{\rnk N}, \label{netbounds1}\\
|\Sigma^{(0)}_{\hU^0}(\pflat,R)|, \,|\Sigma^{(r)}_{\hU^0, \hbu'_1,\dots,\hbu'_r}(\pflat,R)| &\le O(R^2/\pflat^2)^{O_\pol(1)}. \label{netbounds2}
\end{align}

Now let $1\le r\le \rnk$ and  fix distinct $k_1,\dots, k_r\in [0,\rnk]$.
In the sequel we write $I_i= \{ k_1,\dots, k_i\}$.
Let $\cN$ be the set of $\hUU\in \BALL$ such that
\begin{equation}	\label{def:hU0}
\hU^0\in \Sigma_0(\pflat_1,\pssv), \qquad \hbu'_k \in \Sigma_1(\pflat_1) \quad \forall \, k\in I_r
\end{equation}
and
\begin{equation}
\hbu_k' \in \bigg\{ \, (\Slin'\otimes \id_N) \begin{pmatrix} \hu^0_k\\ \hbq_k \end{pmatrix} \,:\, \hbq_k \in \Sigma_{\hU^0, (\hbu'_k : k\in I_r)}(\pflat_1,R)\, \bigg\} \qquad \forall\, k\notin I_r,
\label{hbu.prime}
\end{equation}
with $R=O_\pol(1)$ to be taken sufficiently large, and 
where $\Slin'$ is the matrix $\tSlin^+$ with its first row removed (see \eqref{def:tS+}).

Fix an arbitrary $\UU\in \tbad_{(k_1,\dots, k_r)}(\pflat_0,\pflat_1)\cap \good(\pssv)$.
We claim there exists $\hUU\in \cN$ such that \eqref{UUclose} holds and $\hUU\in \badR_{(k_1,\dots, k_r)}(\pflat_0/2)$. From this the claim will follow by taking $\cN_{(k_1,\dots, k_r)}(\pflat_0,\pflat_1) = \cN\cap  \badR_{(k_1,\dots, k_r)}(\pflat_0/2)$, noting that $\hUU\in \good(\pssv)$ by \eqref{def:hU0}, and that the bound \eqref{netr.size} follows from \eqref{netbounds0}--\eqref{netbounds2}.

We first fix the following submatrices of $\hUU$: take $\hU^0\in \Sigma_0(\pflat_1,\pssv)$ and for each $k\in I_r$ take $\hbu'_k \in \Sigma_1(\pflat_1)$ such that 
\begin{equation}	\label{rank.easy}
\|U^0-\hU^0\|_2\le \pflat_1\quad\text{and}\quad \|\bu'_k - \hbu'_k\|_2 \le \pflat_1 \quad \forall k\in I_r.
\end{equation}
By the triangle inequality, it only remains to choose the vectors $\hbu'_k$ for $k\notin I_r$ so that
\begin{equation}	\label{rank.goal1}
\|\bu'_k - \hbu'_k\|_2 \ls_{\pol,\pssv,\pflat_0} \pflat_1
\end{equation}
and so that the resulting matrix $\hUU$ lies in $\badR_{(k_1,\dots, k_r)}(\pflat_0/2)$.

For arbitrary $k\in [0,\rnk]$, since $\bq_k = (\Slin\otimes \id_N)\bu_k$ and $\|\bu_k\|_2\le2\sqrt{\rnk+1}$, it follows that 
\[
\|\bq_k\|_2 \le 2\sqrt{\rnk+1}\|S\|_\op \ls_\pol1.
\]
For $k\in I_r$, with 
\[
\hbq_{k}:= (S\otimes \id_N)\hbu_{k}, \qquad \hbu_k = \begin{pmatrix} \hu^0_k\\ \hbu'_k \end{pmatrix},
\]
we similarly have
\begin{equation}
\|\bq_{k} - \hbq_{k}\|_2 \ls_\pol \pflat_1 \quad \forall \,k\in I_r.
\end{equation}
Since $\UU\in \tbad_{(k_1,\dots, k_r)}(\pflat_0,\pflat_1)$, 
\begin{equation}	\label{lb:qk1}
 \dist(\bq_{k_i} , \subsp_{I_{i-1}} ) =  \dist(\bq_{k_i} , \subsp^{(i-1)} (U^0,\bu'_{k_1},\dots, \bu'_{k_{i-1}})) \ge \pflat_0\qquad \forall \, 1\le i\le r
\end{equation}
and for any $k\notin I_r$, 
\begin{equation}	\label{ub:qk}
 \dist(\bq_k , \subsp_{I_r} ) = \dist( \bq_k, \subsp^{(r)}(U^0, \bu'_{k_1},\dots,\bu'_{k_r} )) \le \pflat_1
\end{equation}
(recalling \eqref{Wk1}).
Since $U^0\in \good(\pssv)$ we have $\smin(\RR)=\smin(U^0)\ge \alpha$. Together with \eqref{lb:qk1} this implies
\[
\smin( \Flat_{I_{i-1}}) \gs_{\pssv,\pflat_0} 1 \quad \forall \, 1\le i\le r
\]
(where $\Flat_{I_0}:= \RR$). 
This together with \eqref{rank.easy}, \Cref{lem:stab} and \eqref{lb:qk1},
and assuming $c_1(\pol,\pssv,\pflat_0)$ is sufficiently small,
implies 
\begin{equation}	\label{lb:hatq}
  \dist(\hbq_{k_i} , \subsp^{(i-1)} (\hU^0,\hbu'_{k_1},\dots, \hbu'_{k_{i-1}})) \ge \pflat_0 - O_{\pol,\pssv,\pflat_0}(\pflat_1) \qquad \forall \, 1\le i\le r.
\end{equation}
We similarly have
\begin{equation}	\label{ub:hatq}
\dist( \bq_k, \subsp^{(r)}(\hU^0, \hbu'_{k_1},\dots,\hbu'_{k_r} )) \ls_{\pol,\pssv,\pflat_0} \pflat_1
\qquad\forall \,k\notin I_r.
\end{equation}
By the triangle inequality, taking $R=O_\pol(1)$ sufficiently large, there exist $\hbq_k\in \Sigma_{\hU^0,\hbu'_{k_1},\dots, \hbu'_{k_r}}(\pflat_1, R)$ such that
\[
\|\bq_k- \hbq_k\|_2\ls_{\pol,\pssv,\pflat_0} \pflat_1 \qquad \forall \,k\notin I_r.
\]
Finally, setting $\hbu'_k = (\Slin'\otimes\id_N)\hbq_k$, we have for each $k\notin I_r$ that
\begin{align*}
\|\bu'_k - \hbu'_k \|_2 
& = \bigg\| (S'\otimes \id_N) \begin{pmatrix} u^0_k - \hu^0_k \\ \bq_k(\UU) - \hbq_k \end{pmatrix} \bigg\|_2
 \ls_\pol \bigg\| \begin{pmatrix} u^0_k - \hu^0_k \\ \bq_k(\UU) - \hbq_k \end{pmatrix} \bigg\|_2
 \ls_{\pol,\pssv,\pflat_0} \pflat_1
\end{align*}
giving \eqref{rank.goal1} as desired.
That $\hUU\in \badR_{(k_1,\dots, k_r)}(\pflat_0/2)$ follows from \eqref{lb:hatq} and taking the constant $c_1(\pol,\pssv,\pflat_0)$ sufficiently small.

The proof of \Cref{claim:net1} follows similar lines. The net $\cN_\emptyset(\pflat_1)$ is taken as in \eqref{def:hU0}--\eqref{hbu.prime} with $I_r=\emptyset$ (so no columns $\hbu'_k$ are fixed at this initial stage). The estimates \eqref{ub:qk} and \eqref{ub:hatq} are obtained by the same lines, with the subspaces $\subsW_\emptyset = \subsW^{(0)}(U^0)$ and $\subsW^{(0)}(\hU^0)$, while we skip \eqref{lb:qk1} and \eqref{lb:hatq}. 
\end{proof}

\subsection{Proof of \Cref{lem:struct.net}: Constructing a net for structured matrices}
\label{sec:struct.net}

We first record two elementary lemmas. 
The first is a quantitative formulation of the fact that the only way for a vector to be close to all $m$ coordinate hyperplanes generated by a well-conditioned basis of $\C^m$ is for it to have small magnitude. 

\begin{lemma}	\label{lem:smallx}
Let $A\in \Mat_m(\C)$ with $|\det A| \ge \kappa$ and columns $x_1,\dots,x_m\in \ball^m$.
Suppose that for some $v\in \C^m$ we have 
\begin{equation}	\label{hyp:smallx}
|\det((v, A_J))| <\delta \qquad \forall\; J\in {[m]\choose m-1} 
\end{equation}
where we write $A_J$ for the $m\times |J|$ submatrix with columns $\{x_j: j\in J\}$. 
Then $\|v\|_2 < m\delta/\kappa$. 
\end{lemma}

\begin{proof}
We expand $v=\sum_{i=1}^m b_ix_i$. Applying \eqref{hyp:smallx} with $J=[m]\setminus \{j\}$ gives
\[
\delta> |\det((v,A_J)|  = |b_j| |\det A| \ge \kappa |b_j| \qquad \forall\; j\in [m], 
\]
so $\|b\|_2 < \sqrt{m}\delta/\kappa$. 
From our assumptions we have $\|A\|_\op\le \|A\|_\HS\le \sqrt{m}$, and so
\[
\|v\|_2 = \|Ab\|_2 \le \sqrt{m} \|b\|_2 < m\delta/\kappa
\]
as desired.
\end{proof}

\Cref{lem:smallx} will be used in conjunction with the following, which locates a well-conditioned basis of rows in the matrix $U^0$. Recall that $v^{0}_{i}\in\C^\DD$, $i\in [N]$ denote the rows of $U^{0}$.

\begin{lemma}	\label{lem:goodrows}
For any $\UU\in \good(\pssv)$, there exists $I^* = \{ i_0,\dots, i_\rnk\} \subset [N]$ such that 
\begin{equation}	\label{goodrows}
\|v_{i_0}^0\wedge \cdots \wedge v_{i_k}^0\|_2 \ge (\pssv/\sqrt{N})^{k+1}
\end{equation}
for each $0\le k\le \rnk$. 
\end{lemma}

\begin{proof}
We iteratively construct $I^*$ as follows: first, there exists $i_0\in [N]$ such that $\|v_{i_0}^0\|_2 \ge \pssv/\sqrt{N}$, as otherwise $\smin(U^0)\le \|U^0\|_\op\le \|U^0\|_\HS<\pssv$. 
Now for $1\le k\le \rnk-1$, having picked $i_0,\dots, i_k$, there exists $i_{k+1}\in [N]$ such that $\dist(v_{i_{k+1}}^0, \langle v^0_{i_0}, \dots , v^0_{i_k}\rangle ) \ge \pssv/\sqrt{N}$, as otherwise, for any unit vector $u\in \langle v^0_{i_0}, \dots , v^0_{i_k}\rangle^\perp$ we would have
\[
\|U^0u\|_2^2 = \sum_{j\in [N]\setminus \{i_0,\dots, i_k\}} |\langle v_j^{0}, u\rangle|^2 < \pssv^2,
\]
which contradicts $\smin(U^0)\ge \pssv$. \eqref{goodrows} now follows from the base-times-height formula for the norm of the wedge product. 
\end{proof}

Fix an arbitrary $k_{0}\in[\rnk]$. Our proof of \Cref{lem:struct.net} is divided into two cases depending on $\pol$: 
\begin{enumerate}[(A)]
\item \label{caseA} $\slin_{k_0,\ell_0}\ne0$ for some $\ell_0 \in  [\rnk+1,\num]$;
\item \label{caseB} $\slin_{k_0,\ell_0}\ne0$ for some $\ell_0 \in [\rnk]$.
\end{enumerate}
Since the vector $\slin_{k_{0}}\in \C^\num$  has nonzero  norm at least one of these cases must hold. 
Our construction of $\cN$ will be different for each case. 

In what follows, for parameters $R,\rho>0$
we let $\Sigma_0(R,\rho)$ be a $\rho/\sqrt{N}$-net for the ball of radius $R$ in $\C^{[0,\rnk]}$, 
and for $J\subseteq [N]$, we let $\Sigma_J(R,\rho)\subset \matt{[N]}{[0,\rnk]}{\C}$ be a $\rho$-net for the Hilbert--Schmidt ball of radius $R$ on the subspace of matrices supported on the rows indexed by $J$. (We take $\Sigma_\emptyset(R,\rho)$ to consist of the zero matrix.)
As before, all nets are with respect to the appropriate Euclidean metric. 
By \Cref{lem:net} we may take 
\begin{equation}	\label{take.nets}
|\Sigma_0(R,\rho)| = O_{\pol}(R^2N/\rho^2)^{\rnk+1},\qquad  |\Sigma_J(R,\rho)| = O(R^2/\rho^2)^{(\rnk+1)|J|}.
\end{equation}

\begin{proof}[Proof of \Cref{lem:struct.net} under Case \ref{caseA}]
Let $\cN'$ be the set of matrices $\UU= (U^0,\dots, U^\rnk) \in \matt{[0,\rnk]\times [N]}{[0,\rnk]}{\C}$ such that $U^k\in \Sigma_{[N]}(\sqrt{\rnk+1},\pnet)$ for each $k\in [0,\rnk]\setminus\{k_0\}$, and $U^{k_0}$ has rows given by
\begin{equation}	\label{vj.formA}
v_j^{k_0} =- \frac{1}{\slin_{k_0,\ell_0}}  \sum_{k\in [0,\rnk]\setminus\{k_0\}} s_{k,\ell_0} v_j^k \,
, \qquad j\in [N],
\end{equation}
where $v_j^k$ is the $j$th row of $U^k$. 
Now for each $\UU'\in \cN'$ take an element $\UU''\in \cE$ within a distance $2\dist(\UU', \cE)$ of $\UU'$, and let $\cN\subset \cE$ be the set of all $\UU''$ obtained in this way.
By construction we have
\begin{equation}	\label{cNprime.bd}
|\cN| \le |\cN'| \le |\Sigma_{[N]}(\sqrt{\rnk+1},\pnet)|^\rnk = O_\pol(1/\pnet^2)^{\rnk(\rnk+1)N}.
\end{equation}
We claim that for any $\UU\in \cE$ there exists $\hUU\in \cN'$ with 
\begin{equation}	\label{hUU.approx}
\|\UU-\hUU\|_\HS \ls_{\pol,\pssv} \pnet.
\end{equation}
The lemma (under Case \ref{caseA}) clearly follows from this and the triangle inequality, and replacing $\pstr$ above with $c(\pol,\pssv)\pstr$ for a sufficiently small constant $c(\pol,\pssv)>0$.

To show \eqref{hUU.approx}, for arbitrary $\UU\in \cE$, we
take $\hUU\in \cN'$ to be an element such that $\|U^k-\hU^k\|_\HS<\pnet$ for each $k\ne k_0$.
(Recall that $\hU^{k_0}$ is determined by \eqref{vj.formA} once $U^k$ for $k\ne k_0$ are fixed.)
It only remains to show
\begin{equation}	\label{caseA.goal}
\|U^{k_0}-\hU^{k_0}\|_\HS \ls_{\pol, \pssv} \pnet.
\end{equation}
Let $I^*=\{ i_0,\dots, i_\rnk\}\subset [N]$ satisfying \eqref{goodrows} for $\UU$. 
Since $\UU\in \Struct(\pstr)\subset \Struct_1(\pstr)$ we have
\[
| \Delta_{j, I}| := |\det(w_j, v_I^0 )| <\pstr\qquad \forall \;j\in [N], \;\forall\; I\in {I^*\choose \rnk}
\]
where we write $v_I^0:= (v_j^0: j\in I)$, and for ease of notation we write $w_j:= w_j^{\ell_0}$ ($\ell_0$ being fixed at this stage).
From \eqref{goodrows} (with $k=\rnk$) and \Cref{lem:smallx} it follows that
\[
\|w_j\|_2 < \pstr(\rnk+1)(\sqrt{N}/\pssv)^{\rnk+1} \qquad \forall \;j\in [N].
\]
Now for each $j\in [N]$, writing $\hv_j^{k_0}$ for the $j$th row of $\hU^{k_0}$, from \eqref{def:ws}, the triangle and AM-GM inequalities, and the above inequality, we have
\begin{align*}
\|v_j^{k_0} - \hv_j^{k_0} \|_2 
&= \frac1{|s_{k_0,\ell_0}|^2} \Big\| w_j - \sum_{k\in [0,\rnk]\setminus\{k_0\}} s_{k,\ell_0} ( v_j^k- \hv_j^k) \Big\|^2\\
& \ls_{\pol,\pssv} \pstr^2 N^{\rnk+1} + \sum_{k\in [0,\rnk]\setminus\{k_0\}} \|v_j^k - \hv_j^k\|_2^2.
\end{align*}
Summing over $j$ and applying the bounds $\|U^k-\hU^k\|_\HS<\pnet$ for $k\ne k_0$ and our assumed bound on $\delta$, we obtain \eqref{caseA.goal} as desired.
\end{proof}

\begin{proof}[Proof of \Cref{lem:struct.net} under Case \ref{caseB}]
Let $\Lambda_0$ be a $\pnet/\sqrt{N}$-mesh for the interval $[-R,R]$, for some $R=O_\pol(1)$ to be taken sufficiently large.
We take $\cN''$ be the set of matrices $\hUU= (\hU^0,\dots, \hU^\rnk) \in \matt{[0,\rnk]\times [N]}{[0,\rnk]}{\C}$ such that
\begin{equation}	\label{Uk.net}
\hU^k \in \Sigma_{[N]}(\sqrt{\rnk+1}, \pnet) \qquad \forall\;k\in [\rnk]\setminus\{k_0\},
\end{equation}
and for some (possibly empty) $J\subseteq[N]$, 
\begin{align}	\label{Jc.net}
\hU^0_{J^c} &\in \Sigma_{J^c}(\pnet^{1/2}, \pnet)  \quad\text{ and } \quad \hU^{k_0}_{J^c} \in \Sigma_{J^c} (\sqrt{\rnk+1},\pnet) 
\end{align}
while for each $j\in J$, 
\begin{equation}	\label{vj0.net}
\hv_j^0\in \Sigma_0(1,\pfine),\qquad \pfine:= \pnet^{3/2}/N
\end{equation}
and
\begin{equation}	\label{vj.formB}
\hv_j^{k_0} = \frac{1}{\slin_{k_0,\ell_0}} \Big( \hw_j -  \sum_{k\in [0,\rnk]\setminus\{k_0\}} s_{k,\ell_0} \hv_j^k \Big) 
\end{equation}
for some $\hw_j\in \Lambda_0\cdot \langle \hv_j^0\rangle$. 
Here we write $J^c:=[N]\setminus J$, and $\hU^k_{I}$ for the matrix obtained from $\hU^k$ by zeroing out the rows with indices in $I^c$. 

Constructing $\cN\subset \cE$ from $\cN''$ analogously to how $\cN$ was obtained from $\cN'$ in the proof for Case \ref{caseA}, 
we have by construction that 
\begin{align*}
|\cN| 
&\le |\cN''| \\
&\le  |\Sigma_{[N]}(\sqrt{\rnk+1},\pnet)|^{\rnk-1} \sum_{J\subseteq [N]}
|\Sigma_{J^c}(\pnet^{1/2},\pnet)| |\Sigma_{J^c}(\sqrt{\rnk+1},\pnet)| 
|\Sigma_0(1, \pfine)|^{|J|} |\Lambda_0|^{|J|}\\
&= O_\pol(1/\pnet^2)^{(\rnk-1)(\rnk+1)N} \sum_{J\subseteq[N]} O_\pol(1/\pnet)^{(\rnk+1)(N-|J|)} O_{\pol}(1/\pnet^2)^{(\rnk+1)(N-|J|)} \\
&\qquad\qquad\qquad\qquad\qquad\qquad\qquad\times O_{\pol}(1/\pfine^2)^{(\rnk+1)|J|} O_\pol(N/\pnet^2)^{|J|}\\
&= O_\pol(1)^N  O(1/\pnet^2)^{(\rnk+1/2)(\rnk+1)N} \sum_{J\subseteq[N]} O(N)^{(2\rnk+3)|J|} O(1/\pnet^2)^{|J|}\\
&=O_{\pol}(N)^{(2\rnk+3)N}  O(1/\pnet^2)^{(\rnk^2+\frac32\rnk+\frac32)N}
\end{align*}
where in the last line we simply bounded $|J|\le N$ and absorbed the harmless factor $2^N$ in the implied constant. 

As in the proof for Case \ref{caseA}, we will be done if we can show that for every $\UU\in \cE$ there exists $\hUU\in \cN''$ such that \eqref{hUU.approx} holds.
Fixing now an arbitrary $\UU\in \cE$, for each $k\in [\rnk]\setminus \{k_0\}$ we let $\hU^k$ be as in \eqref{Uk.net} with $\|U^k-\hU^k\|_\HS<\pnet$.
It only remains to pick $\hU^0, \hU^{k_0}$ as in \eqref{Jc.net}--\eqref{vj.formB} for some choice of $J=J(\UU)\subseteq [N]$ with
\begin{equation}	\label{caseB.goal1}
\|U^k-\hU^k\|_\HS \ls_{\pol,\pssv} \pnet\,, \quad k\in \{0,k_0\}.
\end{equation}

We take $J$ to be the set of large rows of $U^0$:
\[
J= J(U^0) = \{ j\in [N]: \|v_j^0\|_2\ge \pnet^{1/2}/\sqrt{N}\}.
\]
We have
$
\|U^0_{J^c}\|_\HS < \pnet^{1/2}
$,
and so there exists $\hU^0_{J^c}$ as in \eqref{Jc.net} with $\|U^0_{J^c}-\hU^0_{J^c}\|_\HS\le \pnet$. Completing $\hU^0$ with rows $\hv_j^0$ as in \eqref{vj0.net} such that $\|v_j^0-\hv_j^0\|_2 \le \pfine$ for each $j\in J$, we obtain \eqref{caseB.goal1} for $k=0$. 

Since $\|U^{k_0}\|_\HS \le \sqrt{\rnk+1}$ we can take $\hU^{k_0}_{J^c}$ as in \eqref{Jc.net} with $\|U^{k_0}_{J^c} - \hU^{k_0}_{J^c}\|_\HS \le \pnet$. 
It only remains to show there exist $\hv_j^{k_0}, j\in J$ as in \eqref{vj.formB} such that
\begin{equation}	\label{caseB.goal2}
\sum_{j\in J} \|v_j^{k_0} - \hv_j^{k_0}\|_2^2 \ls_{\pol,\pssv} \pnet^2.
\end{equation}

Turning to this task,
let $I^*=\{ i_0,\dots, i_\rnk\}\subset [N]$ satisfying \eqref{goodrows} for $\UU$. 
For ease of writing we set $m:= \rnk+1$ and denote $x_k:= v_{i_{k-1}}^0$ for $1\le k\le m$. 
We also use the shorthand $x_{\wedge I} :=x_{k_1}\wedge \cdots \wedge x_{k_\ell}$ for $I=\{k_1,\dots, k_\ell\}\subseteq [m]$ with $k_1<\cdots <k_m$. 
From \eqref{goodrows} we have
\begin{equation}	\label{xwedgem}
|x_1\wedge\cdots\wedge x_m| \ge (\pssv/\sqrt{N})^m =: \kappa.
\end{equation}

Consider an arbitrary $j\in J$. 
From \Cref{lem:smallx}, and the bounds \eqref{xwedgem} and $\|v_j^0\|_2\ge \pnet^{1/2}/\sqrt{N}$ it follows that there exists $\wh{I} =\wh{I}_j\subset [m]$ of size $m-1$ such that
\begin{equation}	\label{eyehat}
|v_j^0\wedge x_{\wedge \wh I}| \ge \frac{\kappa\pnet^{1/2}}{m\sqrt{N}}.
\end{equation}
Since $\UU\in \Struct(\pstr)\subset \Struct_2(\pstr)$ we have that for every $I\subset [m]$ of size $m-2$, 
\[	
| \wt\Delta_{j,I}| := |\det(w_j, v_j^0, (x_k)_{k\in I} )| = | w_j\wedge v_j^0 \wedge x_{\wedge I}|<\pstr.
\]
Letting $\tw_j, \tx_1,\dots, \tx_m$ denote the projections of $w_j, x_1,\dots, x_m$ to $\langle v_j^0\rangle^\perp$,
we have in particular that
\begin{equation}	\label{wtD}
\pstr> |w_j\wedge v_j^0 \wedge x_{\wedge I}| = \|v_j^0\|_2 \|\tw_j \wedge \tx_{\wedge I}\|_2  
\ge \frac{\pnet^{1/2}}{\sqrt{N}}\|\tw_j \wedge \tx_{\wedge I}\|_2  \quad \forall\; I\in {\wh{I}\choose m-2}.
\end{equation}
On the other hand, since $\|v_j^0\|_2\le 1$, we have from \eqref{eyehat} that 
\[
\|x_{\wedge \wh{I}}\|_2 \ge  |v_j^0 \wedge x_{\wedge \wh{I}}| \ge \frac{\kappa\pnet^{1/2}}{m\sqrt{N}}.
\]
Together with \eqref{wtD} and \Cref{lem:smallx} (identifying the subspace $\langle v_j^0\rangle^\perp$ with $\C^{m -1}$) this implies
\[
\dist(w_j, \langle v_j^0\rangle) = \|\tw_j\|_2 < \frac{m(m-1)\pstr N}{\kappa \pnet} 
\ls_{\pol,\pssv} \pstr N^{(\rnk+3)/2} / \pnet \ls \pnet/\sqrt{N}
\]
using our assumption on $\pnet$ and recalling $m=\rnk+1=O_\pol(1)$. 
Now from the estimates 
\[
\|w_j\|_2 = O_\pol(1)\,, \qquad \|v_j^0\|_2 \ge \frac{\pnet^{1/2}}{\sqrt{N}}\,, \qquad \|v_j^0-\hv_j^0\|_2 \le  \rho_1= \pnet^{3/2}/N,
\]
the identity
\[
\dist(w_j, \langle v_j^0\rangle) = \frac{\|w_j\wedge v_j^0\|_2}{\|v_j^0\|_2}
\]
and multilinearity of the wedge product, we have
\[
\dist(w_j, \langle \hv_j^0\rangle)  = \dist(w_j, \langle v_j^0\rangle)  + O_{\pol,\pssv}( \pnet/\sqrt{N})  \ls_{\pol,\pssv} \pnet/\sqrt{N}.
\]
Since $\|w_j\|_2=O_\pol(1)$, by taking $R$ sufficiently large we have that for every $j\in J$ there exists $\hw_j\in \Lambda_0\cdot \hv_j^0$ such that
\begin{equation}	\label{wj.approx}
\|w_j-\hw_j\|_2 \ls_{\pol,\pssv} \pnet/\sqrt{N}.
\end{equation}
Taking $\hv_j^{k_0}$ as in \eqref{vj.formB} for each $j\in J$ (having by now fixed all vectors on the right hand side), we have by \eqref{def:ws} and the triangle and AM-GM inequalities
\[
\|v_j^{k_0} - \hv_j^{k_0}\|_2^2 \ls_{\pol} \|w_j - \hw_j\|_2^2 + \sum_{k\in [\rnk]\setminus \{k_0\}} \|v_j^k-\hv_j^k\|_2^2 .
\]
Substituting \eqref{wj.approx}, summing over $j\in J$, and applying the bounds $\|U^k-\hU^k\|_\HS\ls_\pol \pnet$ for $k\ne k_0$ yields \eqref{caseB.goal2} and completes the proof.
\end{proof}

\begin{remark}
Let us explain why we cannot take any linearization for $\pol$. Consider a linearization of the form 
\[ 
\LLz = \begin{pmatrix} -z + T & R_1 & \ldots & R_{n_0} \\ Y_1 & -\id & \ldots& 0\\ \vdots & &\ddots&  \\ {Y}_{n_0} & \cdots& & -\id \end{pmatrix}
\]
(compare \eqref{def:LLz})
where the $R_k$, $Y_\ell$ and $T$ are linear forms in the matrices $X_1,\dots, X_n$ with respective coefficient vectors $r_k, s_\ell, t\in \C^n$.
If the family $\{ r_1,...,r_{n_0} \}$ is not of full rank, there exists a non-null vector $x \in \C^{n_0}$ such that $\sum_{i=1}^{n_0} x_i r_i =0$, then if $\tilde{x} =(0,x_1,...,x_{n_0})$, the matrices $L_i$ as in \eqref{def:Li} (with $\langle r_k,\bxi_i\rangle$ in place of $\xi_i^k$) all satisfy $L_i \tilde{x} =0$. When we look for the coefficients of the terms of order $n_0 +1$ of $\det ( M + \sum_{i=1}^N U^*_i L_i )$, we find that these are the same as in $\det(\sum_{i=1}^N U^*_i L_i )$, and hence they are zero. 
On the other hand, if $\{ s_1,...,s_{n_0} \}$ is not of full rank, we then have a non-null vector $x \in \C^{n_0}$ such that 
$\sum_{k=1}^{n_0} x_k Y_k =0$ 
and so $\sum x_k \row_1(Y_k') =0$. 
Then considering the vector $u$ with first block $u^0=0$ and $u^k(1)=x_k$ for $k =1,...,n_0$, 
$u^k(i) = 0$ for all $k$ and all $i\in [2,N]$,
we have that $u$ is a non-null element of $(\subsV^z_{(1)})^\perp$, so that $\sigma_{min}(U^0) =0$. And so we cannot have our lower bound on the effective rank $\RR(U^0)$ (as in \eqref{def:flat}) and our anti-concentration bound does not beat the cardinality of the net. The problem of finding a ``nice'' linearization for $\pol$ is also why the question of polynomials of degrees higher than $2$ is more difficult. For instance, it is not clear that a homogeneous polynomial of degree 3 admits a linearization of the form $\pol = \sum_{i=1}^{n_0} R_i S_i T_i$ where the families $\{ R_i\}_{1 \leq i \leq n_0 }$ and $\{ T_i\}_{1 \leq i \leq n_0 }$ are of full rank. 
\end{remark}

\section{Proof of  \Cref{thm:brown} }
\label{sec:law}
The proof of  \Cref{thm:brown} follows Girko's idea \cite{Girko} based on Green's formula  \eqref{Green}. 
For  $z \in \C$ , we set
\[
\nu_N^z := \mu_{(z \id_N - \Pol^N)(z \id_N-\Pol^N)^*}\,,
\]
where we recall that $\Pol^N = \pol(X_1^N,\dots, X_n^N)$,
as well as our notation \eqref{def:ESD}.
Applying the above formula to the eigenvalues of $\Pol^N$, we deduce that
\begin{equation}\label{green}\int_\C \psi(z)d\mu_{\Pol^{N}}(z)=\frac{1}{4\pi}\int_\C \Delta \psi(z) \int _0^\infty\log (x) d\nu_N^z(x) dz\,.\end{equation}
The proof  of the convergence of the right hand side is broken into the following steps.

\begin{enumerate}
\item\label{point1} For all $z \in \C$, we show that $\nu_N^{z}$ converges weakly almost surely to some identifiable probability measure $\nu^z$ on $\R_+$.
\item\label{point2} Using our lower bound in \Cref{thm:pseudo}, we show that for almost every $z \in \C$, $ \int_{\R_+} \log (x) d\nu_N^{z}(x)$ converges to $\int_{\R_+} \log (x) d\nu^{z}(x)$ in probability. 
\item\label{point3} We show that, in probability, $ z \mapsto  \int_{\R_+} \log (x) d\nu_N^{z}(x)$ converges in $L^1$ to $ z\mapsto \int_{\R_+} \log (x) d\nu^{z}(x)$, and therefore $\mu_{\Pol^N}$ converges in distribution to a limit that we identify as $\nu_{\pol(c_1,\dots, c_n)}$.
\end{enumerate}

\subsubsection*{Proof of (\ref{point1}):} Because $(z \id_N - \Pol^N)(z \id_N-\Pol^N)^*$ is a self-adjoint polynomial in independent Ginibre matrices and their adjoints, 
the first point is a direct consequence of Voiculescu's theorem \cite{V91}, see \cite{AGZ} for a review.  
In particular we have that $\nu^z$ is the distribution of $|z-\pol(c_1,\dots,c_n)|^2$ in the sense of $*$-moments, and from the boundedness of the circular elements it follows that $\nu^z$ is compactly supported for every fixed $z\in \C$.

\subsubsection*{Proof of (\ref{point2}):} To prove the second point, we need to deal with the fact that the logarithm is unbounded. To this end, first observe
that  by Lemma \ref{lem:norm}, there is a constant $\bound >0$ such that 
\begin{equation}\label{boundP}\limsup_{N \to \infty}  \|\Pol^N\| \leq \bound \qquad a.s.\end{equation}
Therefore, fixing $\epsilon>0$, the first point implies that for any smooth nonnegative function $\chi_\epsilon$ which vanishes on $[0,\epsilon/2]$ and equals one on $[\epsilon, \infty)$,
\begin{equation}\label{conv1}
\lim_{N\rightarrow\infty} \int_0^\bound \chi_\epsilon(x)\log(x) d\nu_N^z(x)=\int_0^\bound \chi_\epsilon(x)\log(x) d\nu^z(x)\qquad a.s.
\end{equation}
The main point is therefore to show that $\int_0^\epsilon \log(x) d\nu_N^z(x)$ is negligible.  We will show that
for every $\delta,\delta' > 0$ there is $\epsilon\in (0,1)$ such that 
\[ 
\limsup_{N \to \infty} \pr\bigg\{\, \bigg|\int_0^{\epsilon} \log x d\nu_N^z(x) \bigg| \geq \delta \,\bigg\} \leq \delta' .
\]
Denoting $\mathcal{G}_N = \{ \|\Pol^N\| \leq \bound, \smin(z-\Pol^N) \geq N^{-\beta}\}$, by Theorem \ref{thm:pseudo} and \eqref{boundP}, we can choose $\beta$ large enough so that $\pr(\cG_N)$
goes to one. Hence, it 
suffices to show that  
\[ 
\lim_{\epsilon \to 0} \limsup_{N \to \infty} \E \Big( \mathds{1}_{\mathcal{G}_N} \int_0^{\epsilon} |\log x| d\nu^{z}_N(x) \Big) =0.
\]
On the event $\mathcal{G}_N$,
\[  
\int_0^{\epsilon} \log x d\nu_N^z(x) = \int_{N^{-\beta}}^{\epsilon} \log x d\nu_N^z(x) ,
\]
so we only need to show 
\begin{equation}\label{cvb}
\lim_{\epsilon \rightarrow 0} \limsup_{n \to \infty} \E \Big|\int_{N^{-\beta}}^{\epsilon} \log x d\nu_N^z(x)\Big| =0. 
\end{equation}
We denote $\bar{\nu}_n^z = \E[ \nu_N^z]$ and $g$ (resp.\ $g_N$)  the Stieljes transform of $\nu^z$ (resp.\ $\bar\nu_N^z$) given for $\zeta\in\mathbb C$ by 
\[
g^z(\zeta)=\int_0^\infty \frac{1}{\zeta-x}d\nu^z(x),\quad g_N^z(\zeta)=\mathbb E[\int_0^\infty\frac{1}{\zeta-x}d\nu_N^z(x)]\,.
\]
The next lemma is the key to prove \eqref{cvb}.

\begin{lemma}\label{lemB} Let $z\in\mathbb C$ be fixed. 
There exist $C,N_0 > 0 $ and $c_1,c_2\in (0,1)$ such that for $\eta\in [N^{-c_1}, 1]$ and $N \geq N_0$,
\[ 
|\Im (g_N^z(i \eta))| \leq C \eta^{-c_2} .
\]
\end{lemma}
We postpone the proof of this lemma to deduce first \eqref{cvb}.
This lemma implies that  for 
$x\in [N^{-c_1},1]$,
\[
F_N^z(x):=\bar{\nu}_N^z([0,x]) \leq 2  \int_{[0,x]}  \frac{x^2}{y^2+x^2}d\bar{\nu}_N^z(y)\le 2x \Im (g_N^z(i x)) \leq 2 C x^{1 - c_2}.
\]
Hence, we find for $\alpha < c_1$ and $\epsilon < 1$: 
\begin{align*}
&\E\Big|\int_{N^{-\beta}}^{\epsilon} \log x d\nu_N^z(x)\Big| \\
&\qquad \le - \int_{N^{-\beta}}^{N^{-\alpha}} \log x d\bar{\nu}_N^z(x)  - \int_{N^{-\alpha}}^{\epsilon} \log x d\bar{\nu}_N^z(x) \\
&\qquad \le \beta(\log N) \bar{\nu}_N^z([0, N^{-\alpha}])+
 \int_{N^{-\alpha}}^{\epsilon} \frac{1}{x} F_N^z(x) dx -  \log(\epsilon) F_N^z(\epsilon) +F_N^z(N^{-\alpha}) \log(N^{-\alpha}) \\
&\qquad \leq 
2C\beta(\log N) N^{-\alpha(1-c_1)}  
+2 C \frac{\epsilon^{1-c_2}}{1-c_2} - 2 (\log\epsilon) \epsilon^{1-c_2} ,
\end{align*}
which gives \eqref{cvb}.

\begin{proof}[Proof of \Cref{lemB}]
First we observe that Haagerup and Thorbj{\o}rnsen \cite{HT} proved the convergence of the Stieltjes transform of $\nu_N^z$ close to the real axis. Indeed, recall that Ginibre matrices can be decomposed as the sum of two independent GUE matrices: $X_j^N=(Y_j^N+iZ_j^N)/\sqrt{2}$ where $(Y_j,Z_j)_{1\le j\le n}$ are independent GUE matrices. Hence $(z-\Pol^N)(z-\Pol^N)^*$ can be seen as a polynomial in independent GUE matrices so that Haagerup and Thorbj{\o}rnsen result applies and, see e.g 
 \cite[Lemma 5.5.4]{AGZ}, implying that there exists $c_1$ finite such that  for $\Im\zeta\in [N^{-c_1},  1]$ and $N$ large enough,
\[ 
|g^z(\zeta) - g_N^z(\zeta)| \leq \frac{c_2}{N^2 (\Im \zeta)^{c_3}} .
\]
Up to take a smaller $c_1$, it is therefore enough to show that 
 $$\Im g^z( i \epsilon) \leq K \epsilon^{q}$$ 
 for some $q>-1$ and $K > 0$. Following \cite[Corollary 1.2]{ShlSkou},  $\nu^z$ has no atoms. Moreover, by \cite[Theorem 1.1]{ShlSkou}, $g^z$ is bounded close to the real line except possibly on a discrete set $A$. Assuming at worst that $A$ contains the origin, the same theorem shows that there exists $q\in\mathbb Q$ and a constant $K\neq 0$  such that $g^z(\zeta)\simeq K\zeta^q$ for $\zeta$ close to the origin. But clearly, since $\nu^z$ has no atoms, $q>-1$. A more quantitative proof could have used that \cite{BM} implies that the partition function of $\nu^z$ is H\"older with exponent $1/15$. 
 \end{proof}

\subsubsection*{Proof of (\ref{point3}):}
Denoting
\[
h_N(z) := \int_0^\infty \log|x| d \nu_N^z(x)\quad\mbox{ and }\quad h(z) =\int_0^\infty \log|x| d \nu^z(x) 
\]
we have shown that for every fixed $z\in \C$, $h_N(z)$ converges in probability to $h(z)$, 
the latter now being well defined by steps (\ref{point1}) and (\ref{point2}). In particular we have 
\[
h(z)= \int_\C\log|z-\lambda|d\nu_{\pol(c_1,\dots, c_n)}(\lambda)
\]
by definition of the Brown measure.
We next prove that for any compact set $K$, and on the events $\mathcal{A}_N = \{ \|\Pol^N\| < \bound \}$, $h_N$ converges as well in $L^1$ in the sense that
\begin{equation}\label{der}
\lim_{N \to \infty}  \E\Big( \mathds{1}_{\mathcal{A}_N}\int_{z \in K}| h_N (z) -h(z)| dz \Big) =0 .
\end{equation}
This is enough to conclude the  proof of  \Cref{thm:brown} by \eqref{green} for any twice continuously differentiable function $\psi$ with compact support.  The last condition is finally removed since the eigenvalues are almost surely bounded by $B$ according to \eqref{boundP}. To prove \eqref{der}, it is enough to notice that $h_N$ and $h$ belong to $L^2$ in the sense that
$$\mathbb E[ 1_{\mathcal{A}_N}\int_K |h_N(z)|^2dz]+\int_K|h(z)|^2 dz $$
is bounded independently of $N$, so that the bounded convergence theorem applies. But, Jensen's inequality and Fubini's theorem imply that
$$\mathbb E[ 1_{\mathcal{A}_N}\int_K |h_N(z)|^2dz]\le \mathbb E[1_{\mathcal{A}_N}\int_\C \int_K|\log|z-\lambda||^2 dz d\mu_{\Pol^N}(\lambda)]\le \sup_{|\lambda|\le B} \int_K|\log|z-\lambda||^2 dz$$ 
is finite, and a similar estimate holds for $h(z)$ (one obtains from the boundedness of the circular elements $c_1,\dots, c_n$ that $\nu_{\pol(c_1,\dots,c_n)}$ has compact support). Therefore, {in probability $h_N$ converge to $h$ in $L^1(K,\Leb)$} and so $\mu_{\Pol^N} = \frac1{4\pi}\Delta h_N$ converges to 
$\nu_{\pol(c_1,\dots, c_m)} = \frac1{4\pi} \Delta h$
in the sense of distributions on $K$. Again, taking $K$ that contains the support of 
$\nu_{\pol(c_1,\dots, c_m)} $
and $\mu_{\Pol^N}$,  the convergence in the sense of distributions implies weak convergence and the result is proved.

\bibliographystyle{plain}
\bibliography{bibliodeg}

\end{document}